\documentclass[]{article}

			\usepackage{amssymb,amsmath,amsthm}
			\usepackage{mathtools}
			\usepackage{xspace}
			\usepackage{mathrsfs}
			\usepackage[shortlabels]{enumitem}

			\usepackage{hyperref}\hypersetup{hidelinks}
			
			\usepackage{bm}
			\theoremstyle{plain}
			\newtheorem{theorem}{Theorem}
			\newtheorem{lemma}{Lemma}
			\newtheorem{proposition}{Proposition}
			\theoremstyle{definition}
			\newtheorem{definition}{Definition}
			\numberwithin{theorem}{section}
			\numberwithin{lemma}{section}
			\numberwithin{proposition}{section}
			\numberwithin{definition}{section}
			\numberwithin{equation}{section}
			\let\withqed\qed
			\let\qed\relax
			\DeclarePairedDelimiter{\brax}{(}{)}
			\DeclarePairedDelimiter{\sqbrax}{[}{]}
			\DeclarePairedDelimiter{\setbrax}{\{}{\}}
			\DeclarePairedDelimiter{\set}{\{}{\}}
			\DeclarePairedDelimiter{\abs}{\lvert}{\rvert}
			
			\DeclarePairedDelimiter{\normDoubleBar}{\lVert}{\rVert}
			\DeclarePairedDelimiter{\ceil}{\lceil}{\rceil}
			\DeclarePairedDelimiter{\floor}{\lfloor}{\rfloor}
			\newcommand{\open}[2]{\brax{#1 , #2}}
			\newcommand{\clsd}[2]{\sqbrax{#1 , #2}}
			\newcommand{\openclsd}[2]{(#1 , #2]}
			\newcommand{\clsdopen}[2]{[#1 , #2)}
			\newcommand{\Ocv}[2]{O_{#1}}
			\newcommand{\llcv}[2]{\ll_{#1}}
			\newcommand{\ggcv}[2]{\gg_{#1}}
			\newcommand{\suchthat}{:}
			\renewcommand{\vec}[1]{\bm{#1}}
			\newcommand{\grad}{\vec{\nabla}}
			\makeatletter
			\newsavebox\myboxA
			\newsavebox\myboxB
			\newlength\mylenA
			\newcommand*\widebar[1]{%
				\sbox{\myboxA}{$\m@th#1$}%
				\setbox\myboxB\null
				\ht\myboxB=1.05\ht\myboxA%
				\dp\myboxB=\dp\myboxA%
				\wd\myboxB=0.9\wd\myboxA
				\sbox\myboxB{$\m@th\overline{\copy\myboxB}$}
				\setlength\mylenA{\the\wd\myboxA}
				\addtolength\mylenA{-\the\wd\myboxB}%
				\ifdim\wd\myboxB<\wd\myboxA%
				\rlap{\hskip 0.5\mylenA\usebox\myboxB}{\usebox\myboxA}%
				\else
				\hskip -0.5\mylenA\rlap{\usebox\myboxA}{\hskip 0.5\mylenA\usebox\myboxB}%
				\fi}
			\makeatother
			\newcommand{\tbbQ}[0]{$\mathbb{Q}$\xspace}
			\newcommand{\bbC}[0]{\mathbb{C}}
			\newcommand{\bbN}[0]{\mathbb{N}}
			\newcommand{\bbP}[0]{\mathbb{P}}
			\newcommand{\bbQ}[0]{\mathbb{Q}}
			\newcommand{\bbR}[0]{\mathbb{R}}
			\newcommand{\bbZ}[0]{\mathbb{Z}}
			\newcommand{\calI}[0]{{\mathcal I}}
			\newcommand{\calM}[0]{{\mathcal M}}
			\newcommand{\frakb}[0]{\mathfrak{b}}
			\newcommand{\frakm}[0]{\mathfrak{m}}
			\newcommand{\frakB}[0]{\mathfrak{B}}
			\newcommand{\frakI}[0]{\mathfrak{I}}
			\newcommand{\frakM}[0]{\mathfrak{M}}
			\newcommand{\frakS}[0]{\mathfrak{S}}
			\newcommand{\tfrakb}[0]{$\mathfrak{b}$\xspace}
			\newcommand{\tfrakB}[0]{$\mathfrak{B}$\xspace}
			\newcommand{\cancellation}{{\mathscr C}}
			
			\newcommand{\rank}{\operatorname{rank}}
			\newcommand{\sing}{\operatorname{Sing}}
			\newcommand{\meas}[1]{\lambda\set{#1}}
			\newcommand{\Meas}[1]{\lambda\set[\big]{#1}}
			\newcommand{\supnorm}[1]{\normDoubleBar{#1}_\infty}
			\newcommand{\supnormbig}[1]{\normDoubleBar[\big]{#1}_\infty}
			\newcommand{\tee}{t} 
			\newcommand{\inradius}{r_1}
			\newcommand{\outradius}{r_2}
			\newcommand{\localarcs}[2]{\set{\tfrac{\vec{a}}{q}\in \bbQ^R\cap\clsdopen{0}{1}^R\suchthat q \leq #1, \abs{\Sloc{q}{\vec{a}}}\geq #2 }}
			\newcommand{\weightbox}{\mathscr{B}}
			\newcommand{\expSumSBoxBigAt}[1]{S\brax[\big]{#1 ; P}}
			\newcommand{\expSumSBoxAt}[1]{S\brax{#1 ; P}}
			\newcommand{\expSumSBox}{\expSumSBoxAt{\vec{\alpha}}}
			\newcommand{\majorarcs}{\frakM_{P,d,\Delta}}
			\newcommand{\tmajorarcs}{$\majorarcs$\xspace}
			\newcommand{\minorarcs}{\frakm_{P,d,\Delta}}
			
			\newcommand{\singSeries}{\frakS}
			\newcommand{\singSeriesIncomplete}[1]{\frakS(#1)}
			\newcommand{\singIntegralBox}{\frakI}
			\newcommand{\singIntegralBoxIncomplete}[1]{\frakI(#1)}
			\newcommand{\Sloc}[2]{S_{#1}(#2)}
			\newcommand{\gradFMultilinear}[2]{\vec{m}^{(#1 \cdot \vec{f})} \brax{ #2 }}
			\newcommand{\gradSomethingMultilinear}[2]{\vec{m}^{(#1)} \brax{ #2 }}
			\newcommand{\vecsuper}[2]{\vec{#1}^{(#2)}}
			\newcommand{\numSolnsInBox}{N_{f_1,\dotsc,f_R}}
			\newcommand{\numZeroesInBoxOf}[1]{N_{#1}}
			\newcommand{\weylDiffIneqNumSolns}{U_{\vec{\beta}\cdot\vec{f}} }
			\newcommand{\weylDiffIneqOfSomethingNumSolns}[1]{U_{#1}}
			\newcommand{\auxIneqOfSomethingNumSolns}[1]{N^{\operatorname{aux}}_{#1}}
			\newcommand{\auxIneqNumSolns}{N^{\operatorname{aux}}_{\vec{\beta}\cdot\vec{f}}}
			\newcommand{\betaDotCapitalF}{\vec{\beta}\cdot\vec{F}}
			\newcommand{\betaDotLeadingPart}{\vec{\beta}\cdot\vec{f}^{[d]}}
			\newcommand{\aDotLeadingPart}{\vec{a}\cdot\vec{f}^{[d]}}
			\newcommand{\aDotCapitalF}{\vec{a}\cdot\vec{F}}
			\newcommand{\aOverQDotLittleF}{\tfrac{\vec{a}}{q}\cdot\vec{f}}
			\newcommand{\alphaDotLittleF}{\vec{\alpha}\cdot\vec{f}}
			\newcommand{\alphaDotLeadingPart}{\vec{\alpha}\cdot\vec{f}^{[d]}}
			\newcommand{\gammaDotLeadingPart}{\vec{\gamma}\cdot\vec{f}^{[d]}}

\begin{document}

\title{Quadratic forms and systems of forms in many variables\thanks{This research was supported by Engineering and Physical Sciences Research Council grants EP/J500495/1 and EP/M507970/1.}
}
\author{\href{https://maths.fan}{Simon L. Rydin Myerson}
}

\maketitle

\begin{abstract}
Let $F_1,\dotsc,F_R$ be quadratic forms with integer coefficients in $n$ variables. When  $n\geq 9R$ and the variety $V(F_1,\dotsc,F_R)$ is a smooth complete intersection, we prove an asymptotic formula for the number of integer points in an expanding box at which these forms simultaneously vanish, which in particular implies the Hasse principle for $V(F_1,\dotsc,F_R)$. Previous work in this direction required $n$ to grow at least quadratically with $R$. We give a similar result for $R$ forms of degree $d$, conditional on an upper bound for the number of solutions to an auxiliary inequality. In principle this result may apply as soon as $n> d2^dR$. In the case that $d\geq 3$, several strategies are available to prove the necessary upper bound for the auxiliary inequality. In a forthcoming paper we use these ideas to apply the circle method to nonsingular systems of forms with real coefficients.

quadratic forms $\cdot$ rational points

\end{abstract}

\section{Introduction}

\subsection{Results}\label{1.sec:main_result}

Our goal is to improve the following classic result of Birch.

\begin{theorem}[Birch~\cite{birchFormsManyVars}]\label{1.thm:birch's_theorem_long}
	Let $d\geq 2$ and let $F_1(\vec{x}),\dotsc,F_R(\vec{x})$ be homogeneous forms of degree $d$, with integer coefficients in $n$ variables $x_1,\dotsc,x_n$. Let $\weightbox$ be a box in $\bbR^n$, contained in the box $\clsd{-1}{1}^R$, and having sides of length at most 1 which are parallel to the coordinate axes. For each $ P\geq 1$, write
	\begin{equation*}
	\numZeroesInBoxOf{F_1,\dotsc,F_R}(P)
	=
	\#
	\set{ \vec{x}\in\bbZ^n \suchthat
		\vec{x}/P\in\weightbox,\,
		F_1(\vec{x})=0, \dotsc,F_R(\vec{x})=0
	}.
	\end{equation*}
	Let  $W$ be the projective variety cut out in $\bbP_\bbQ^{n-1}$ by the condition that the $R\times n$ Jacobian matrix $\brax{\partial F_i(\vec{x})/ \partial x_j}_{ij}$ has rank less than $R$. If
	\begin{equation}\label{1.eqn:birch's_condition_long}
	n-1-\dim W
	> (d-1)2^{d-1}R(R+1),
	\end{equation}
	then for all $P\geq 1$, some  $\singIntegralBox\geq 0$ depending only on the $c_i$ and $\weightbox$, and some $\singSeries\geq 0$ depending only on the $c_i$, we have
	\begin{equation}\label{1.eqn:HL_formula}
	\numZeroesInBoxOf{F_1,\dotsc,F_R}(P)
	=
	\singIntegralBox\singSeries P^{n-dR}
	+
	O\brax{P^{n-dR-\delta}}
	\end{equation}
	where the implicit constant depends only on the forms $F_i$ and $\delta$ is a positive constant depending only on $d$ and $R$. If the variety $V(F_1,\dotsc,F_R)$ cut out in $\bbP_\bbQ^{n-1}$ by the forms $F_i$ has a smooth point over $\bbQ_p$ for each prime $p$ then $\singSeries >0$, and if it has a smooth real point whose homogeneous co-ordinates lie in $\weightbox$ then  $\singIntegralBox>0$.
\end{theorem}

Here $\singIntegralBox,$ $\singSeries$ are the usual singular integral and series; see \eqref{1.eqn:evaluating_frakI} and \eqref{1.eqn:evaluating_frakS} below. 

We focus in particular on weakening the hypothesis \eqref{1.eqn:birch's_condition_long} on the number of variables, when the number of forms $R$ is greater than one.  Previous improvements of this type have required $R=1$ or 2. Our first result, proved in \S\ref*{1.sec:main_thm_proof}, is as follows:

\begin{theorem}\label{1.thm:main_thm_short}
	When $d=2$ and $\dim V(F_1,\dotsc,F_R)=n-1-R$, we may replace \eqref{1.eqn:birch's_condition_long} with the condition
	\begin{equation}
	n-\sigma_\bbR
	>
	8R,
	\label{1.eqn:condition_on_n_short}
	\end{equation}
	where $\sigma_\bbR$ is the element of $\set{0,\dotsc,n}$ defined by		\begin{equation}\label{1.eqn:def_of_sigma-sub-R}
	\sigma_\bbR=
	1+\max_{\vec{\beta}\in\bbR^R\setminus\set{\vec{0}}} \dim\sing V(\betaDotCapitalF),
	\end{equation}
	and $ V(\betaDotCapitalF)$ is the of the hypersurface  cut out in  $\bbP_\bbR^{n-1}$ by $\beta_1F_1+\dotsc+\beta_RF_R=0$.
\end{theorem}

Note that \eqref{1.eqn:condition_on_n_short} is equivalent to
\begin{equation}\label{1.eqn:condition_on_rank}
\min_{\vec{\beta}\in\bbR^R\setminus\set{\vec{0}}} \rank (\betaDotCapitalF) > 8R,
\end{equation}
where $\rank (\betaDotCapitalF)$ is the rank of the matrix of the quadratic form $ \beta_1F_1+\dotsc+\beta_RF_R$. The hypothesis \eqref{1.eqn:condition_on_n_short}  is strictly weaker than the case $d=2$ of the condition \eqref{1.eqn:birch's_condition_long} as soon as $R\geq 4$. Indeed we have $\sing V(\betaDotCapitalF)\subset W$ whenever $\vec{\beta}\in\bbR^R\setminus\set{\vec{0}}$, and so
\[
\sigma_\bbR
\leq
1+\dim W.
\]
Thus \eqref{1.eqn:condition_on_n_short}  is  weaker than \eqref{1.eqn:birch's_condition_long} whenever $2R(R+1)<8R$ holds, that is for $R\geq 4$.

To obtain the result described in the abstract we can simplify \eqref{1.eqn:condition_on_n_short} with the following lemma, proved at the end of \S\ref*{1.sec:main_thm_proof}.	
\begin{lemma}\label{1.lem:nonsing_case}
	Let $d\geq 2$ and let $F_1,\dotsc,F_R$ and $W$ be as in Theorem~\ref*{1.thm:birch's_theorem_long}. If $V(F_1,\dotsc,F_R)$ is smooth with dimension $n-1-R$, then we have
	\begin{equation}\label{1.eqn:nonsing_case}
	\sigma_\bbR
	\leq
	1+\dim W
	\leq R-1.
	\end{equation}
\end{lemma}

If $V(F_1,\dotsc,F_R)$ is a smooth complete intersection and $n\geq 9R$ then Theorem~\ref*{1.thm:main_thm_short} and Lemma~\ref*{1.lem:nonsing_case} imply that the asymptotic formula \eqref{1.eqn:HL_formula} holds. This in turn implies that $V(F_1,\dotsc,F_R)$ satisfies the Hasse principle, by the last part of Theorem~\ref*{1.thm:birch's_theorem_long}. As is usual with the circle method one also obtains weak approximation for $V(F_1,\dotsc,F_R)$ in this case; see the comments after the proof of Theorem~\ref*{1.thm:main_thm_short} in \S\ref*{1.sec:main_thm_proof}.

The ``square-root cancellation" heuristic discussed around formula (1.12) in Browning and Heath-Brown~\cite{browningHeathBrownDiffDegrees} suggests that the condition $n > 4R$ should suffice in place of the $n\geq 9R$ in the previous paragraph. So \eqref{1.eqn:condition_on_n_short} brings us within a constant factor of square-root cancellation as $R$ grows, while \eqref{1.eqn:birch's_condition_long} misses by a factor of $O(R)$.


We deduce  Theorem~\ref*{1.thm:main_thm_short} from the following more general result, proved in \S\ref*{1.sec:main_thm_proof}.

\begin{definition}\label{1.def:aux_ineq}
	For each $k \in\bbN\setminus\set{\vec{0}}$ and $\vec{t}\in\bbR^k$ we write  $\supnorm{\vec{t}} = \max_i \abs{t_i}$ for the supremum norm. Let $f(\vec{x})$ be any polynomial of degree $d\geq 2$ with real coefficients in $n$ variables $x_1,\dotsc,x_n$.  For $i= 1,\dotsc, n$ we define
	\begin{equation*}
	m^{( f )}_i \brax{\vec{x}^{(1)},\dotsc,\vec{x}^{(d-1)} }
	=
	\sum_{j_1,\dotsc,j_{d-1}=1}^n
	x^{(1)}_{j_1} \dotsm x^{(d-1)}_{j_{d-1}}
	\frac{\partial^{d} f(\vec{x})}{\partial x_{j_1} \dotsm \partial x_{j_{d-1}} \partial x_i},
	\end{equation*}
	where we write $\vecsuper{x}{j}$  for a vector of $n$ variables $(x^{(j)}_1,\dotsc,x^{(j)}_n)^T$. This defines an $n$-tuple of multilinear forms
	\[
	\gradSomethingMultilinear{ f } {\vec{x}^{(1)},\dotsc,\vec{x}^{(d-1)} }\in \bbR[\vec{x}^{(1)},\dotsc,\vec{x}^{(d-1)}]^n.
	\]
	Finally, for each  $B \geq 1$  we put $\auxIneqOfSomethingNumSolns{f} \brax{ B }$ for the number of $(d-1)$-tuples of integer $n$-vectors $\vec{x}^{(1)}, \dotsc, \vec{x}^{(d-1)}$ with
	\begin{gather}
	\supnorm{\vecsuper{x}{1}},\dotsc,\supnorm{\vecsuper{x}{d-1}} \leq B, 
	\nonumber
	\\
	\label{1.eqn:aux_ineq}
	\supnorm{\gradSomethingMultilinear{ f }{ \vec{x}^{(1)}, \dotsc, \vec{x}^{(d-1)} }} < \supnorm{ f^{[d]} } B^{d-2}
	\end{gather}
	where we let $\supnorm{f^{[d]}} = \frac{1}{d!} \max_{\vec{j}\in\set{1,\dotsc,n}^d} \abs[\big]{\frac{\partial^d f(\vec{x})}{\partial x_{j_1}\dotsm\partial x_{j_d}}}$.
\end{definition}

\begin{theorem}\label{1.thm:manin}
	Let the forms $F_i$ and the counting function $	\numZeroesInBoxOf{F_1,\dotsc,F_R}(P)$ be as in Theorem~\ref*{1.thm:birch's_theorem_long}, and let $\auxIneqOfSomethingNumSolns{f}(B)$ be as in Definition~\ref*{1.def:aux_ineq}. Suppose that the $F_i$ are linearly independent and that
	\begin{equation}\label{1.eqn:aux_ineq_bound_in_manin_thm}
	\auxIneqOfSomethingNumSolns{\betaDotCapitalF}(B)
	\leq
	C_0 B^{(d-1)n-2^d\cancellation}
	\end{equation}
	for some $C_0 \geq 1$, $\cancellation > dR$ and all $\vec{\beta}\in\bbR^R$ and $B\geq 1$, where we have written $\betaDotCapitalF$ for $ \beta_1F_1+\dotsb+\beta_RF_R$. Then for all $P\geq 1$ we have
	\begin{equation*}
	\numZeroesInBoxOf{F_1,\dotsc,F_R}(P)
	=
	\singIntegralBox\singSeries P^{n-dR}
	+
	O\brax{P^{n-dR-\delta}},
	\end{equation*}
	where the implicit constant depends at most on $C_0$, $\cancellation$ and the $F_i$, and $\delta$ is a positive constant depending at most on $\cancellation$, $d$ and $R$. Here  $\singIntegralBox$ and $\singSeries $ are as in Theorem~\ref*{1.thm:birch's_theorem_long}.
\end{theorem}

One trivially has \[
B^{(d-2)n}
\ll_{d,n}
\auxIneqOfSomethingNumSolns{\betaDotCapitalF}(B)\ll_{d,n}B^{(d-1)n}.
\]
So \eqref{1.eqn:aux_ineq_bound_in_manin_thm} requires us to save a factor of $P^{2^d\cancellation}$ over the trivial upper bound, while the largest saving possible is of size $O(P^n)$. It follows that we  must have $n>d2^{d}R$ in order for both \eqref{1.eqn:aux_ineq_bound_in_manin_thm} and $\cancellation>dR$ to hold.

Counting functions similar to $\auxIneqOfSomethingNumSolns{\betaDotCapitalF}(B)$ play a similar role in some other applications of the circle method, with the equations
\begin{equation}\label{1.eqn:aux_eq}
{\gradSomethingMultilinear{ f }{ \vec{x}^{(1)}, \dotsc, \vec{x}^{(d-1)} }}
= \vec{0}
\end{equation}
in place of the inequality \eqref{1.eqn:aux_ineq}. The quantities $M(a_1,\dotsc,a_r;H)$ from formula (9) of Dietmann~\cite{dietmannWeylsIneq}, and $\calM_f(P)$ from Lemma~2 of Schindler~\cite{schindlerWeylsIneq} are both of this type. In this setting one needs to save a factor of size $B^{O(R^2)}$ over the trivial bound.

In forthcoming work we bound the function $\auxIneqOfSomethingNumSolns{\betaDotCapitalF}(B)$ for degrees higher than 2, with the goal of handling systems $F_i$ in roughly $d2^{d}R$ variables. We will approach this problem variously by using elementary methods, by generalising the argument used in Lemma~3 of Davenport~\cite{davenportSixteen} to treat the equations \eqref{1.eqn:aux_eq}, and by applying the circle method iteratively to the inequalities \eqref{1.eqn:aux_ineq}. We will also combine the ideas used here with the variant of the circle method due to Freeman \cite{freemanAsymptBoundsAndFormulas} to give a version of Theorem~\ref*{1.thm:manin} for systems of forms $F_i$ with real coefficients.

\subsection{Related work}\label{1.sec:related_work}

Theorem~1 of M\"uller~\cite{mullerSystemsQuadIneqsAndValueDistns} gives a result with exactly the same number of variables as Theorem~\ref*{1.thm:main_thm_short}, but for quadratic inequalities with real coefficients rather than quadratic equations with rational coefficients. It is in turn founded on work of Bentkus and G\"otze~\cite{bentkusGotzeEllipsoids,bentkusGotzeDistributionQuadForms} concerning a single quadratic inequality. The method of proof is related to ours, see \S\S\ref*{1.sec:moat_lemmas} and~\ref*{1.sec:weyl_diff} below.


When $d=2$, the forms $F_i$ are diagonal and the variety $V(F_1,\dotsc,F_R)$ is smooth, then the conclusions of Theorem~\ref*{1.thm:birch's_theorem_long} hold whenever $n > 4R$. That is, we have the ``square-root cancellation" situation described at the end of \S\ref*{1.sec:main_result}. This follows by standard methods from a variant of Hua's lemma due to Cook~\cite{cookNoteHuasLemma}.

When $d=2$ Dietmann~\cite{dietmannSystemsQuadForms}, improving work of Schmidt~\cite{schmidtSystemsQuadForms}, gives conditions similar to \eqref{1.eqn:condition_on_n_short} under which the asymptotic formula \eqref{1.eqn:HL_formula} holds and the constant $\singSeries$ is positive. In particular it is sufficient that either $
\min_{\vec{\beta}\in\bbC^R\setminus\set{\vec{0}}} \rank(\betaDotCapitalF)>2R^2+3R,
$ or that $\min_{\vec{a}\in\bbQ^R\setminus\set{\vec{0}}} \rank(\aDotCapitalF)
>2R^3 +\tau(R)R$, where $\tau(R) =2$ if $R$ is odd and 0 otherwise. He also shows that if $d=2$, the variety $V(F_1,\dotsc,F_R)$ has a smooth real point and
$
\min_{\vec{a}\in\bbQ^R\setminus\set{\vec{0}}} \rank(\aDotCapitalF)
>2R^3-2R 
$
then $V(F_1,\dotsc,F_R)$ has a rational point.

Munshi~\cite{munshiPairsQuadrics11Vars} proves the asymptotic formula \eqref{1.eqn:HL_formula} when $d=2$, $n=11$ and $V(F_1,F_2)$ is smooth. By contrast using Theorem~\ref*{1.thm:birch's_theorem_long}  and \eqref{1.eqn:nonsing_case} would require $n \geq 14$. When $d=2$ and $R=1$ we have a single quadratic form $F$. Heath-Brown~\cite{heathBrownNewForm} then proves such an asymptotic formula  whenever $V(F)$ is smooth and $n \geq 3$.

If $F$ is a cubic form, Hooley \cite{hooleyOctonaryCubicsII} shows that when $n=8$, the variety $V(F)$ is smooth, and $\weightbox$ is a sufficiently small box centred at a point where the Hessian determinant of $F$ is nonzero, then we have a smoothly weighted  asymptotic formula analogous to  \eqref{1.eqn:HL_formula}. This result is conditional on a Riemann hypothesis for a certain modified Hasse-Weil $L$-function. For $n=9$ he proves a similar result without any such assumption \cite{hooleyNonaryCubicsIII}, with an error term $O(P^{n-3}(\log P)^{-\delta})$ instead of the $O(P^{n-3-\delta})$ in \eqref{1.eqn:HL_formula}. In this setting Theorem~\ref*{1.thm:birch's_theorem_long} requires $n \geq 17$.

In the case of a single quartic form $F$ such that $V(F)$ is smooth, Hanselmann~\cite{hanselmannQuartics40Vars} gives the condition $n\geq 40$ in place of the $n \geq 49$ required to apply Theorem~\ref*{1.thm:birch's_theorem_long}. Work in progress of Marmon and Vishe yields a further improvement.

When $d\geq 5$ and $R=1$, a sharper condition than \eqref{1.eqn:birch's_condition_long} is available by work of Browning and Prendiville~\cite{browningPrendivilleImprovements}. For $d\leq 10$ and a smooth hypersurface $V(F)$ this is essentially a reduction of one quarter in the number of variables required.

Dietmann~\cite{dietmannWeylsIneq} and Schindler~\cite{schindlerWeylsIneq} show that the condition \eqref{1.eqn:birch's_condition_long} may be replaced with $n-\sigma_\bbZ>(d-1)2^{d-1}R(R+1)$, where we define
\begin{equation}\label{1.eqn:def_of_sigma-sub-Z}
\sigma_\bbZ
=
1+ \max_{\vec{a}\in \bbZ^R\setminus\set{ \vec{0} }} \dim\sing V(\aDotLeadingPart),
\end{equation}	
Note that the maximum here is over integer points, and so we may have $\sigma_\bbZ < \sigma_\bbR$.

Birch's work~\cite{birchFormsManyVars} is generalised to systems of forms with differing degrees by Browning and Heath-Brown~\cite{browningHeathBrownDiffDegrees} over \tbbQ and by Frei and Madritsch~\cite{freiMadritschDifferingDegrees} over number fields. It is extended to linear spaces of solutions by Brandes~\cite{brandesFormsRepresentingForms,brandesLinearSpacesNumberFields}. Versions of the result for function fields are due to Lee~\cite{leeFunctionFields} and to Browning and Vishe~\cite{browningVisheFunctionFields}. A version for bihomogeneous forms is due to Schindler~\cite{schindlerBihomogeneous}, and  Mignot~\cite{mignotTridegree111,mignotCircleMethodWithToricHeights} further develops these methods for certain trilinear forms and for hypersurfaces in toric varieties. Liu~\cite{liuQuadricsPrimes} proves existence of solutions in prime numbers to a quadratic equation in 10 or more variables. Asymptotic formulae for systems of equations of the same degree with prime values of the variables are considered by Cook and Magyar~\cite{cookMagyarPrimeVars} and by Xiao and Yamagishi~\cite{xiaoYamagishiPrimeVars}. Magyar and Titchetrakun~\cite{magyarEtAlAlmostPrimes} extend these results to values of the variables with a bounded number of prime factors, while Yamagishi~\cite{yamagishiDifferingDegreesPrimes} treats systems of equations with differing degrees and prime variables. It is natural to ask whether similar generalisations exist for Theorem~\ref*{1.thm:main_thm_short}.

\subsection{Notation}\label{1.sec:notation}

Parts of our work apply to polynomials with general real coefficients. Therefore we let $f_1(\vec{x}),\dotsc,f_R(\vec{x})$ be polynomials with real coefficients, of degree $d \geq 2$ in $n$ variables $x_1,\dotsc,x_n$, and we write $f^{[d]}_1(\vec{x}), \dotsc, f^{[d]}_R(\vec{x})$ for the degree $d$ parts.

Implicit constants in $\ll$ and big-$O$ notation are always permitted to depend on the polynomials $f_i$, and hence on $d,n,$ and $R$. We use scalar product notation to indicate linear combinations, so that for example $\alphaDotLittleF=\sum_{i=1}^R \alpha_i f_i$. Throughout, $\supnorm{\vec{t}}$, $\supnorm{f}$, $\vecsuper{m}{f}$ and $\auxIneqOfSomethingNumSolns{f} \brax{ B }$ are as in Definition~\ref*{1.def:aux_ineq}. We do not require algebraic varieties to be irreducible, and we use the convention that $\dim\emptyset = -1$. 

By an \emph{admissible box} we mean a box in $\bbR^n$  contained in the box $\clsd{-1}{1}^R$, and having sides of length at most 1 which are parallel to the coordinate axes. We let $\weightbox$ be an admissible box. For each $\vec{\alpha}\in\bbR^R$ and $P\geq 1$, we define the exponential sum
\begin{equation}\label{1.eqn:def_of_S}
\expSumSBox
=
\sum_{\substack{ \vec{x} \in \bbZ^n \\ \vec{x}/P \in \weightbox }}
e( \alphaDotLittleF(\vec{x}) )
\end{equation}
where $e(\tee) = e^{2\pi i \tee}$. This depends implicitly on $\weightbox$ and the  $f_i$. We often write the expression $\max\setbrax{P^{-d} \supnorm{\vec{\beta}}^{-1}, \supnorm{\vec{\beta}}^{\frac{1}{d-1}}}$, and if $\vec{\beta}=\vec{0}$ this quantity is defined to be $+\infty$.

\subsection{Structure of this paper}

In \S\ref*{1.sec:circle_method} we apply the circle method to a system of degree $d$ polynomials with integer coefficients, assuming a certain hypothesis \eqref{1.eqn:assumed_arcs} on $\expSumSBox$. In \S\ref*{1.sec:aux_ineq} we prove this hypothesis on ${\expSumSBox}$ for polynomials with real coefficients, assuming that the bound \eqref{1.eqn:aux_ineq_bound_in_manin_thm} above holds. We then prove Theorems~\ref*{1.thm:main_thm_short} and~\ref*{1.thm:manin} in \S\ref*{1.sec:main_thm_proof}.

\section{The circle method}\label{1.sec:circle_method}

In this section we apply the circle method, assuming that the bound
\begin{equation}\label{1.eqn:assumed_arcs}
\min\setbrax*{ \abs*{\frac{\expSumSBox}{P^{n+\epsilon}}} , \,  \abs*{\frac{\expSumSBoxAt{\vec{\alpha}+\vec{\beta}}}{P^{n+\epsilon}}} }
\leq
C
\max \setbrax[\big]{  P^{-d}\supnorm{\vec{\beta}}^{-1}  ,\,   \supnorm{\vec{\beta}}^{\frac{1}{d-1}}  }^{\cancellation}
\end{equation}
holds for all $\vec{\alpha},\vec{\beta}\in \bbR^R$, $P \geq 1$, some $\cancellation>dR$, $C \geq 1$, and some small $\epsilon>0$. In particular we will show that \eqref{1.eqn:assumed_arcs} implies that the set of points $\vec{\alpha}$ in $\bbR^R$ where $\abs{\expSumSBox}$ is large has small measure. Our goal is the  result below, which will be proved in \S\ref*{1.sec:completing_the_circle_method}.

\begin{proposition}\label{1.prop:circle_method}
	Assume that the polynomials $f_i$ have integer coefficients, and that the leading forms $f^{[d]}_i(\vec{x})$ are linearly independent. Write
	\begin{equation}\label{1.eqn:def_of_num_solns_in_box}
	\numSolnsInBox(P)
	=
	\#
	\set{ \vec{x}\in\bbZ^n \suchthat
		\vec{x}/P\in\weightbox,\,
		f_1(\vec{x})=\dotsb= f_R(\vec{x})=\vec{0}
	}.
	\end{equation}
	Suppose we are given $\cancellation>dR$, $C \geq 1$ and $\epsilon>0$ such  that  the bound \eqref{1.eqn:assumed_arcs} holds  for all $\vec{\alpha},\vec{\beta}\in \bbR^R$, all $P \geq 1$ and all admissible boxes $\weightbox$. If  $\epsilon$ is sufficiently small in terms of $\cancellation$, $d$ and $R$, then we have
	\begin{equation*}
	\numSolnsInBox(P)
	=
	\singIntegralBox\singSeries P^{n-dR}
	+
	O_{C,f_1,\dotsc,f_R}\brax{P^{n-dR-\delta}}
	\end{equation*}
	for all $P\geq 1$, all admissible boxes $\weightbox$, and some  $\delta>0$ depending only on $\cancellation$, $d$, $R$. Here $\singIntegralBox,$ $\singSeries$ are the usual singular integral and series given by  \eqref{1.eqn:evaluating_frakI} and \eqref{1.eqn:evaluating_frakS} below.
\end{proposition}

We comment on the role of \eqref{1.eqn:assumed_arcs}. If the  $f_i$ have integer coefficients, then we have
\begin{equation}\label{1.eqn:circle_method}
\numSolnsInBox(P)
=
\int_{\clsd{0}{1}^R} \expSumSBox \,\mathrm{d}\vec{\alpha}.
\end{equation}
If both $\expSumSBox$ and $\expSumSBoxAt{\vec{\alpha}+\vec{\beta}}$ are large then  \eqref{1.eqn:assumed_arcs} implies that one of the terms $P^{-d}\supnorm{\vec{\beta}}^{-1}$ or $\supnorm{\vec{\beta}}^{\frac{1}{d-1}}$ must be large. In particular, the points $\vec{\alpha}$ and $\vec{\alpha}+\vec{\beta}$ must either be  very close  or somewhat far apart. In this sense \eqref{1.eqn:assumed_arcs} is a ``repulsion principle" for the sum $\expSumSBox$. We can use this fact to bound the measure of the set where $\expSumSBox$ is large, and this will enable us to reduce \eqref{1.eqn:circle_method} to an integral over major arcs.

To see the source of the condition $\cancellation>dR$ in Proposition~\ref*{1.prop:circle_method}, consider the case
\begin{equation}\label{1.eqn:amount_of_cancellation}
\abs{\expSumSBox} = \abs{\expSumSBoxAt{\vec{\alpha}+\vec{\beta}}} = CP^{n-\cancellation+\epsilon}.
\end{equation}
In general we always have
\begin{equation*}
\max \setbrax[]{  P^{-d}\supnorm{\vec{\beta}}^{-1}  ,\,   \supnorm{\vec{\beta}}^{\frac{1}{d-1}}  }^\cancellation \geq P^{-\cancellation},
\end{equation*}
with equality when $\supnorm{\vec{\beta}} = P^{1-d}$ holds. So in the case \eqref{1.eqn:amount_of_cancellation}, the assumption \eqref{1.eqn:assumed_arcs} is trivial. In other words \eqref{1.eqn:assumed_arcs} might still be satisfied even if the function $\expSumSBox$ had absolute value $P^{n-\cancellation+\epsilon}$ at every point $\vec{\alpha}$ in real $R$-space. This will lead to an error term of size at least $P^{n-\cancellation+\epsilon}$ in evaluating the integral \eqref{1.eqn:circle_method}. Hence we require $\cancellation>dR$ in the proposition above in order for the error term to be smaller than the main term.

\subsection{Mean values from bounds of the form \eqref{1.eqn:assumed_arcs}}\label{1.sec:moat_lemmas}

We show that the bound \eqref{1.eqn:assumed_arcs} implies upper bounds for the integral of the function $\expSumSBox$ over any bounded measurable set. M\"uller~\cite{mullerSystemsQuadIneqsAndValueDistns} and Bentkus and G\"otze~\cite{bentkusGotzeEllipsoids,bentkusGotzeDistributionQuadForms} previously used similar ideas to treat quadratic forms with real coefficients.

We begin with a technical lemma.

\begin{lemma}\label{1.lem:from_moats_to_mean_values}
	Let  $\inradius:\open{0}{\infty}\to\open{0}{\infty}$ be a strictly decreasing  bijection, and let $\outradius:\open{0}{\infty}\to\open{0}{\infty}$ be a strictly increasing  bijection. Write $\inradius^{-1}$ and $\outradius^{-1}$ for the inverses of these maps. Let $\nu>0$ and let $E_0$ be a hypercube in $\bbR^R$ whose sides are of length $\nu$ and parallel to the coordinate axes. Let $E$ be a measurable subset of $E_0$ and let $\varphi: E\to\clsdopen{0}{\infty}$ be a measurable function.
	
	Suppose that for all $\vec{\alpha},\vec{\beta}\in\bbR^R$ such that $\vec{\alpha}\in E$ and $\vec{\alpha}+\vec{\beta}\in E$, we have
	\begin{equation}\label{1.eqn:arcs_for_random_function}
	\min\setbrax*{ \varphi(\vec{\alpha}) , \,  \varphi(\vec{\alpha}+\vec{\beta}) }
	\leq
	\max \setbrax[\big]{  \inradius^{-1}(\supnorm{\vec{\beta}})  ,\,   \outradius^{-1}(\supnorm{\vec{\beta}})  }.
	\end{equation}
	Then, for any integers $k$ and $\ell$ with  $k<\ell$, we have
	\begin{align}
	\int_{E} \varphi(\vec{\alpha}) \,\mathrm{d}\vec{\alpha}
	\ll_R {}&
	\nu^R2^k
	+
	\sum_{i=k}^{\ell-1}2^i
	\brax[\bigg]{
		\frac{
			\nu\inradius(2^i)}{\min\setbrax{\outradius(2^i),\nu}}}^R
	\nonumber
	\\
	&+
	\brax[\bigg]{\frac{\nu\inradius(2^\ell)}{\min\setbrax{\outradius(2^\ell),\nu}}}^R		\sup_{\vec{\alpha}\in E} \varphi\brax{\vec{\alpha}}
	,
	\label{1.eqn:from_moats_to_mean_values}
	\end{align}
	where the implicit constant depends only on $R$.
\end{lemma}

Note that if we choose
\[
\varphi(\vec{\alpha})=\abs{\expSumSBox} / CP^{n+\epsilon},
\qquad
\inradius(\tee)
= P^{-d}\tee^{-1/\cancellation},
\qquad
\outradius(\tee)
=
\tee^{(d-1)/\cancellation},
\]
then the hypotheses \eqref{1.eqn:assumed_arcs} and \eqref{1.eqn:arcs_for_random_function} become identical. This will enable us to apply Lemma~\ref*{1.lem:from_moats_to_mean_values} to bound the integral $\int_{\minorarcs}\expSumSBox\,\mathrm{d}\vec{\alpha}$, where $\minorarcs$ is a set of minor arcs on which $\expSumSBox$ is somewhat small.

\begin{proof}
	The strategy of proof is as follows. We deduce from \eqref{1.eqn:arcs_for_random_function} that  if both $\varphi(\vec{\alpha})\geq \tee$ and $\varphi(\vec{\alpha}+\vec{\beta})\geq \tee$ hold, then either $\supnorm{\vec{\beta}}\leq \inradius(\tee)$ or $\supnorm{\vec{\beta}}\geq \outradius(\tee)$ must hold. From this we will show that the set of points $\vec{\alpha}$ satisfying the bound $\varphi(\vec{\alpha})\geq \tee$  can be covered by a collection of hypercubes of side $2 \inradius(\tee)$, each of which is separated from the others by a gap of size $\tfrac{1}{2}\outradius(\tee)$. The lemma will follow upon  bounding the total Lebesgue measure of this collection of hypercubes.
	
	For each $\tee>0$ we set
	\begin{equation}
	\label{1.eqn:def_of_superlevel_set_for_random_function}
	D\brax{\tee}
	=
	\set{\vec{\alpha}\in E\suchthat \varphi(\vec{\alpha})\geq \tee}.
	\end{equation}
	Observe that if $\vec{\alpha}$ and $\vec{\alpha}+\vec{\beta}$ both belong to $D\brax{\tee{}}$, then \eqref{1.eqn:arcs_for_random_function} implies that 
	\[
	\max \setbrax[\big]{  \inradius^{-1}(\supnorm{\vec{\beta}})  ,\,   \outradius^{-1}(\supnorm{\vec{\beta}})  }
	\geq
	t,
	\]
	from which it follows that either $\supnorm{\vec{\beta}}\leq \inradius(\tee)$ or $\supnorm{\vec{\beta}}\geq \outradius(\tee)$ must hold.
	
	Let $\frakb$ be any hypercube in $\bbR^R$ whose sides are of length $\frac{1}{2}\outradius(\tee)$  and parallel to the coordinate axes. We claim that $\frakb \cap D\brax{\tee}$ is contained in a hypercube \tfrakB whose sides are of length $2\inradius(\tee)$. To see this let $\vec{\alpha}$ be any fixed vector lying in $\frakb \cap D\brax{\tee}$, and set 
	\begin{equation*}
	\frakB
	=
	\set{\vec{\alpha}+\vec{\beta}\suchthat \vec{\beta}\in\bbR^R, \supnorm{\vec{\beta}}\leq \inradius(\tee)}.
	\end{equation*}
	If  $\vec{\alpha}+\vec{\beta}$ belongs to $\frakb \cap D\brax{\tee}$, then by definition of  \tfrakb the bound $\supnorm{\vec{\beta}} \leq \frac{1}{2}\outradius(\tee)$ must hold. In particular $\supnorm{\vec{\beta}} <\outradius(\tee)$, so by the comments after \eqref{1.eqn:def_of_superlevel_set_for_random_function},  the bound  $\supnorm{\vec{\beta}} \leq \inradius(\tee)$ must hold. This shows that $\vec{\alpha}+\vec{\beta} \in \frakB$, and hence that $\frakb \cap D\brax{\tee}\subset \frakB$, as claimed. In particular the Lebesgue measure of $\frakb \cap D\brax{\tee}$ is at most
	$
	(2\inradius(\tee))^R.
	$
	
	The set $D(\tee)$ is contained in $E_0$, a hypercube of side $\nu$. So in order to cover the set $D(\tee)$  with boxes \tfrakb of side $\tfrac{1}{2}\outradius(\tee)$ one needs  at most
	\[
	\ll_R \frac{\nu^R}{\min\setbrax{\outradius(\tee),\nu}^{R}}
	\]
	boxes. Summing over all the boxes \tfrakb, it follows that
	\begin{equation}\label{1.eqn:measure_of_arcs}
	L\brax{\tee}
	\ll_R
	\brax[\bigg]{\frac{\nu\inradius(\tee)}{\min\setbrax{\outradius(\tee),\nu}}}^R,
	\end{equation}
	where we write $L\brax{\tee}$ for the Lebesgue measure of $D\brax{\tee}$. So we have
	\begin{align*}
	\int_{E} \varphi (\vec{\alpha})\,\mathrm{d}\vec{\alpha}
	={}
	&
	\int\limits_{ E \setminus D\brax{2^k} }
	\varphi(\vec{\alpha})
	\,\mathrm{d}\vec{\alpha}
	+ \sum_{i=k}^{\ell-1}\,
	\int\limits_{ E \cap \brax[]{D\brax{2^{i}} \setminus D\brax{2^{i+1}}}}
	\varphi(\vec{\alpha})
	\,\mathrm{d}\vec{\alpha}
	\nonumber
	\\&
	+
	\int\limits_{E\cap D\brax{2^\ell}}
	\varphi(\vec{\alpha})
	\,\mathrm{d}\vec{\alpha}
	\\
	{}\leq{}
	&
	\nu^R 2^k
	+ \sum_{i=k}^{\ell-1}2^{i+1}
	L\brax{2^i}
	+
	L\brax{2^\ell}
	\sup_{\vec{\alpha}\in E} \varphi(\vec{\alpha}).
	\end{align*}
	With \eqref{1.eqn:measure_of_arcs} this yields \eqref{1.eqn:from_moats_to_mean_values}.
\withqed\end{proof}

We now apply Lemma~\ref*{1.lem:from_moats_to_mean_values} to deduce mean values from bounds of the form \eqref{1.eqn:assumed_arcs}. The following result is stated in greater generality than is strictly required here, to facilitate future applications to forms with real coefficients. 

\begin{lemma}\label{1.lem:mean_value_from_power-law_moats}
	Let $T$ be a complex-valued measurable function on $\bbR^R$. Let $E_0$ be a hypercube in $\bbR^R$ whose sides are of length $\nu$ and parallel to the coordinate axes, and let $E$ be a measurable subset of $E_0$. Suppose that the inequality
	\begin{equation}
	\label{1.eqn:power-law_moats}
	\min\setbrax*{ \abs*{\frac{T\brax{ \vec{\alpha} }}{P^{n}}} , \,  \abs*{\frac{T\brax{ \vec{\alpha}+\vec{\beta} }}{P^{n}}} }
	\leq
	\max \setbrax[\big]{  P^{-d}\supnorm{\vec{\beta}}^{-1} ,\,   \supnorm{\vec{\beta}}^{\frac{1}{d-1}}  }^\cancellation
	\end{equation}
	holds for  some $P\geq 1$ and $\cancellation>0$ and all $\vec{\alpha}, \vec{\beta} \in \bbR^R$. Suppose further that 
	\begin{equation}\label{1.eqn:weyl's_ineq_mean_value_lem}
	\sup_{\vec{\alpha}\in E} \abs{T\brax{\vec{\alpha}}}
	\leq
	P^{n-\delta}
	\end{equation}
	for some $\delta\geq 0$. Then we have
	\begin{multline}\label{1.eqn:mean_value_lem}
	\int_{E} {T\brax{\vec{\alpha}}} \,\mathrm{d}\vec{\alpha}
	\\
	\ll_{\cancellation,d,R}
	\left\{
	\begin{array}{@{}l@{}l@{\hspace{1.7em}}l@{}}
	\nu^R P^{n-\cancellation}
	&{}+
	P^{n-\cancellation-(d-1)R}
	&\text{if }	\cancellation < R
	\\
	\nu^R P^{n-\cancellation}
	&{}+
	P^{n-dR}\log P
	&\text{if } \cancellation = R
	\\
	\nu^R P^{n-\cancellation}
	&{}+
	P^{n-dR-\delta\brax{ 1-\frac{R}{\cancellation}}}
	&\text{if } R < \cancellation < dR
	\\
	\nu^R P^{n-\cancellation}\log P
	&{}+
	P^{n-dR-\delta\brax{ 1-\frac{R}{\cancellation}}}
	&\text{if } \cancellation = dR
	\\
	\nu^R P^{n-dR-\delta(1-\frac{dR}{\cancellation})}
	&{}+
	P^{n-dR-\delta\brax{ 1-\frac{R}{\cancellation}}}
	&\text{if } \cancellation > dR.
	\end{array}
	\right.
	\end{multline}
\end{lemma}

Later we will take $T\brax{\vec{\alpha}} = C^{-1} P^{-\epsilon} \expSumSBox$ where $C$ is as in Proposition~\ref*{1.prop:circle_method}. We will take $E$ to be a set of minor arcs $\minorarcs$, and we will interpret the integral $\int _{\minorarcs}  \expSumSBox\,\mathrm{d}\vec{\alpha}$ as an error term, which will need to be smaller than a main term of size around $P^{n-dR}$. As a result, only the case $\cancellation>dR$ of the bound \eqref{1.eqn:mean_value_lem} will be satisfactory for the present application.

\begin{proof}
	We apply Lemma~\ref*{1.lem:from_moats_to_mean_values} with
	\begin{equation}\label{1.eqn:choice_of_phi,_inradius,_outradius}
	\varphi(\vec{\alpha}) = \frac{\abs{T(\vec{\alpha})}}{P^n},
	\quad
	\inradius(\tee) = P^{-d}\tee^{-1/\cancellation},
	\quad
	\outradius(\tee) = \tee^{(d-1)/\cancellation},
	\end{equation}
	noting that the bound \eqref{1.eqn:arcs_for_random_function} then follows from \eqref{1.eqn:power-law_moats}.
	
	It remains to choose the parameters $k$ and $\ell$ from \eqref{1.eqn:from_moats_to_mean_values}. We will choose these so that the right-hand side of \eqref{1.eqn:from_moats_to_mean_values} is dominated by the sum $\sum_{i=k}^{\ell-1}$, rather than either of the other two terms. More precisely,  take
	\begin{equation}\label{1.eqn:def_of_tee-sub-i}
	k
	=
	\floor{\log_2 P^{-\cancellation}},
	\qquad
	\ell
	=
	\ceil{\log_2 P^{-\delta} },
	\end{equation}
	observing that
	\begin{equation}
	\tfrac{1}{2}P^{-\cancellation}
	<
	2^k
	\leq
	P^{-\cancellation},
	\quad
	P^{-\delta}
	\leq
	2^\ell
	<
	2 P^{-\delta}.
	\label{1.eqn:tee-sub-k_and_tee-sub-ell}
	\end{equation}
	We may assume that $\cancellation>\delta$, for otherwise the bound $\int_E T(\vec{\alpha})\,\mathrm{d}\vec{\alpha} \leq \nu^R P^{n-\delta}$, which follows from \eqref{1.eqn:weyl's_ineq_mean_value_lem}, is stronger than any of the bounds listed in \eqref{1.eqn:mean_value_lem}. We then have $k < \ell$ and so this choice of $k, \ell$ is admissible in 
	Lemma~\ref*{1.lem:from_moats_to_mean_values}. Hence  \eqref{1.eqn:from_moats_to_mean_values} holds, and substituting in our choices \eqref{1.eqn:choice_of_phi,_inradius,_outradius} for the parameters yields
	\begin{align}
	\int_{E} \frac{\abs{T\brax{\vec{\alpha}}}}{P^n} \,\mathrm{d}\vec{\alpha}
	&\ll_R
	\nu^R2^k 
	+
	\sum_{i=k}^{\ell-1}
	2^i
	\brax[\bigg]{\frac{ 
			\nu
			P^{-d}  2^{-{i/\cancellation}}}{ \min\setbrax{2^{{(d-1)i/\cancellation}}, \nu }}}^R
	\nonumber
	\\
	&\hphantom{{}\ll_R{}}+\brax[\bigg]{
		\frac{ 
			\nu
			P^{-d}  2^{-{\ell/\cancellation}} }{ \min\setbrax{2^{{(d-1)\ell/\cancellation}}, \nu }}}^R
	\sup_{\vec{\alpha}\in E}  \frac{\abs{T\brax{\vec{\alpha}}}}{P^n}.
	\label{1.eqn:mean_value_I}
	\end{align}
	By \eqref{1.eqn:weyl's_ineq_mean_value_lem}  and \eqref{1.eqn:tee-sub-k_and_tee-sub-ell} we have $\sup_{\vec{\alpha}\in E}  \frac{\abs{T\brax{\vec{\alpha}}}}{P^n}\leq 2^\ell $, and so we may extend the sum in \eqref{1.eqn:mean_value_I} from $\sum_{i=k}^{\ell-1}$ to $\sum_{i=k}^\ell$ to obtain
	\begin{equation*}
	\int_{E} \frac{\abs{T\brax{\vec{\alpha}}}}{P^n} \,\mathrm{d}\vec{\alpha}
	\ll_R
	\nu^R 2^k
	+
	\sum_{i=k}^{\ell}
	2^i
	\brax[\bigg]{\frac{ 
			\nu
			P^{-d}  2^{-{i/\cancellation}} }{ \min\setbrax{2^{{(d-1)i/\cancellation}}, \nu } }}^R.
	\end{equation*}
	Since
	\[
	\frac{ P^{-d}  2^{-{i/\cancellation}}}{ \min\setbrax{2^{{(d-1)i/\cancellation}}, \nu }}
	\leq
	P^{-d}2^{-di/\cancellation}+\nu^{-1}P^{-d}2^{-i/\cancellation},
	\]
	we deduce that
	\begin{equation}\label{1.eqn:mean_value_II}
	\int_{E} \frac{\abs{T\brax{\vec{\alpha}}}}{P^n} \,\mathrm{d}\vec{\alpha}
	\ll_R
	\nu^R2^k +
	\sum_{i=k}^{\ell}
	\nu^R
	P^{-dR}    2^{i(1-dR/\cancellation)}
	+
	\sum_{i=k}^{\ell}
	P^{-dR}   2^{i(1-R/\cancellation)}.
	\end{equation}
	Note that
	\begin{equation*}
	\sum_{i=k}^{\ell}
	2^{i(1-dR/\cancellation)}
	\ll_{\cancellation,d,R}
	\begin{cases}
	2^{k\brax{1-dR/\cancellation}}
	&\text{if }\cancellation<dR
	\\
	\ell-k
	&\text{if }\cancellation = dR
	\\
	2^{\ell\brax{1-dR/\cancellation}}
	&\text{if }\cancellation>dR.
	\end{cases}
	\end{equation*}
	Recall from \eqref{1.eqn:tee-sub-k_and_tee-sub-ell} that we have $2^k \geq \tfrac{1}{2} P^{-\cancellation}$ and $2^\ell \leq 2 P^{-\delta}$, and observe that by \eqref{1.eqn:def_of_tee-sub-i} the bound $\ell-k \leq 2+\cancellation\log_2 P$ holds. It follows that 
	\begin{equation*}
	\sum_{i=k}^{\ell}
	2^{i(1-dR/\cancellation)}
	\ll_{\cancellation,d,R}
	\begin{cases}
	P^{\cancellation-dR}
	&\text{if }\cancellation<dR
	\\
	\log P
	&\text{if }\cancellation = dR
	\\
	P^{-\delta\brax{1-dR/\cancellation}}
	&\text{if }\cancellation>dR,
	\end{cases}
	\end{equation*}
	and reasoning similarly for $\sum_{i=k}^{\ell}
	2^{i(1-R/\cancellation)}$, we deduce from \eqref{1.eqn:mean_value_II}  that
	\begin{multline*}
	\int_{E} \frac{\abs{T\brax{\vec{\alpha}}}}{P^n} \,\mathrm{d}\vec{\alpha}
	\\
	\ll
	\left\{
	\begin{array}{@{}l@{}l@{\hspace{1.7em}}l@{}}
	\nu^R2^k+
	\nu^R  P^{-\cancellation}
	&{}+
	P^{n-\cancellation-(d-1)R}
	&\text{if }	\cancellation < R
	\\
	\nu^R2^k+
	\nu^R  P^{-\cancellation}
	&{}+
	P^{-dR}\log P
	&\text{if } \cancellation = R
	\\
	\nu^R2^k+
	\nu^R  P^{-\cancellation}
	&{}+
	P^{-dR-\delta(1-R/\cancellation)}
	&\text{if } R < \cancellation < dR
	\\
	\nu^R2^k+
	\nu^R P^{-\cancellation}\log P
	&{}+
	P^{-dR-\delta(1-R/\cancellation)}
	&\text{if } \cancellation = dR
	\\
	\nu^R2^k+
	\nu^R P^{-dR-\delta(1-dR/\cancellation)}
	&{}+
	P^{-dR-\delta(1-R/\cancellation)}
	&\text{if } \cancellation > dR,
	\end{array}
	\right.
	\end{multline*}
	with an implicit constant depending only on $\cancellation,d,$ and $R$. One final application of the bound $2^k \leq P^{-\cancellation}$ from \eqref{1.eqn:tee-sub-k_and_tee-sub-ell} completes the proof of \eqref{1.eqn:mean_value_lem}.
\withqed\end{proof}

\subsection{Notation for the circle method}\label{1.sec:circle_method_notation}

We split the domain $\clsd{0}{1}^R$ into two regions. Let $\Delta \in \open{0}{1}$ and set
\begin{align}\label{1.eqn:def_of_major_arcs}
\majorarcs
&=
\bigcup_{\substack{q \in \bbN \\ q\leq P^\Delta}}
\bigcup_{\substack{0 \leq a_1,\dotsc,a_R \leq q\\ (a_1,\dotsc ,a_R,q) =1}}
\set[\big]{\vec{\alpha}\in\clsdopen{0}{1}^R \suchthat \supnormbig{\vec{\alpha}-\tfrac{\vec{a}}{q}} <  P^{\Delta-d} },
\\
\nonumber
\minorarcs
&=
\clsd{0}{1}^R\setminus\majorarcs.
\end{align}
We give local analogues of $\expSumSBox$ and of the integral $\int_{\majorarcs}\expSumSBox\,\mathrm{d}\vec{\alpha}$. We set
\begin{equation*}
\Sloc{q}{\vec{a}}
=
q^{-n}
\sum_{\vec{y}\in \set{1,\dotsc,q}^n}
e\brax[\big]{\aOverQDotLittleF(\vec{y})}
\end{equation*}
for each $q \in \bbN$ and $\vec{a}\in\bbZ^R$, and we put
\begin{equation*}
\singSeriesIncomplete{P}
=
\sum_{q\leq P^\Delta}
\sum_{\substack{\vec{a}\in\set{1,\dotsc,q}^R \\ (a_1,\dotsc ,a_R,q) =1}}
\Sloc{q}{\vec{a}}.
\end{equation*}
For each $\vec{\gamma}\in\bbR^R$, set
\begin{equation*}
\Sloc{\infty}{\vec{\gamma}}
=
\int_{\weightbox}
e\brax[\big]{\gammaDotLeadingPart(\vec{\tee})}
\,\mathrm{d}\vec{\tee},
\end{equation*}
and let
\begin{equation*}
\singIntegralBoxIncomplete{P}
=
\int\limits_{\substack{\vec{\alpha}\in\bbR^R \\ \supnorm{\vec{\alpha}}\leq P^{\Delta-d} }} P^n \Sloc{\infty}{P^d\vec{\alpha}}  \,\mathrm{d}\vec{\alpha}.
\end{equation*}
Finally we define a quantity $\delta_0$ which in some sense measure the extent to which the system $f_i$ is  singular. Let  $\sigma_\bbZ\in \set{ 0,\dotsc,n }$ be as in \eqref{1.eqn:def_of_sigma-sub-Z}, and let
\begin{equation*}
\delta_0
=
\frac{n-\sigma_\bbZ}{(d-1)2^{d-1}R}.
\end{equation*}

\subsection{The minor arcs}

On the minor arcs $\minorarcs$ we have the following bound, compare \eqref{1.eqn:weyl's_ineq_mean_value_lem} in Lemma~\ref*{1.lem:mean_value_from_power-law_moats}.

\begin{lemma}[Dietmann~\cite{dietmannWeylsIneq}, Schindler~\cite{schindlerWeylsIneq}]\label{1.lem:weyl's_ineq}
	Suppose that the polynomials $f_i$ have integer coefficients. Let $\Delta$, $\minorarcs$ and $\delta_0$ be as in \S\ref*{1.sec:circle_method_notation}, and let $\epsilon>0$. Let the sum $\expSumSBox$ be as in \eqref{1.eqn:def_of_S}. 
	Then we have
	\begin{equation}\label{1.eqn:weyl's_ineq}
	\sup_{\vec{\alpha}\in\minorarcs} \abs{\expSumSBox}
	\llcv{\epsilon}{d,n,R,\epsilon}
	P^{n-\Delta\delta_0 + \epsilon}	\end{equation}
	where the implicit constant depends only on $d,n, R,$ and $\epsilon$. The constant $\delta_0$ satisfies $\delta_0 \geq \frac{1}{(d-1)2^{d-1}R}$ whenever the forms $f^{[d]}_i$ are linearly independent. 
\end{lemma}

\begin{proof}
	The bound \eqref{1.eqn:weyl's_ineq}  follows either from Lemma~4 in Dietmann~\cite{dietmannWeylsIneq}, or from Lemma~2.2 in Schindler~\cite{schindlerWeylsIneq}, by setting the parameter $\theta$ in either author's work to be
	\begin{equation*}
	\theta
	=
	\frac{\Delta-\epsilon}{(d-1)R},
	\end{equation*}
	and taking $P\ggcv{\epsilon}{d,n,R,\epsilon} 1$ sufficiently large. Provided the forms $f^{[d]}_i$ are linearly independent, the variety $V(\aDotLeadingPart)$ is a proper subvariety of $\bbP_\bbQ^{n-1}$  for  each $\vec{a}\in\bbZ^R\setminus\set{\vec{0}}$, and so  $\sigma_\bbZ \leq n-1$ holds, by \eqref{1.eqn:def_of_sigma-sub-Z}. This implies that  $\delta_0 \geq \frac{1}{(d-1)2^{d-1}R}$, as claimed.
\withqed\end{proof}

\subsection{The major arcs}\label{1.sec:major_arcs}

In this section we estimate $\int_{\majorarcs}\expSumSBox\,\mathrm{d}\vec{\alpha}$, the integral over the major arcs.

\begin{lemma}\label{1.lem:Sloc}
	Suppose that the polynomials $f_i$ have integer coefficients. Let $\Delta$, \tmajorarcs, $\Sloc{\infty}{\vec{\gamma}}$, $\Sloc{q}{\vec{a}}$, $\singSeriesIncomplete{P}$ and $\singIntegralBoxIncomplete{P}$ be as in \S\ref*{1.sec:circle_method_notation}. Then for all $\vec{a}\in\bbZ^R$ and all $q\in\bbN$ such that $q \leq P$, we have
	\begin{equation}\label{1.eqn:Sloc}
	\expSumSBoxBigAt{\tfrac{\vec{a}}{q}+\vec{\alpha}}
	=
	P^n\Sloc{q}{\vec{a}}  \Sloc{\infty}{P^d\vec{\alpha}}
	+
	\Ocv{}{\vec{f}}\brax[]{  qP^{n-1}\brax{ 1+P^{d}\supnorm{\vec{\alpha}} }},
	\end{equation}
	and it follows that
	\begin{equation}\label{1.eqn:integral_over_frakM}
	\int_{\majorarcs} \expSumSBox \,\mathrm{d}\vec{\alpha}
	=
	\singSeriesIncomplete{P}\singIntegralBoxIncomplete{P}
	+ \Ocv{}{\vec{f}} \brax[\big]{ P^{n-dR+(2R+3)\Delta -1}}.
	\end{equation}
\end{lemma}

\begin{proof}		
	To show \eqref{1.eqn:Sloc} we follow the proof of Lemma~5.1 in Birch~\cite{birchFormsManyVars}.	First observe that $\alphaDotLittleF(\vec{x}) = \alphaDotLeadingPart(\vec{x}) +O(\supnorm{\vec{x}}^{d-1}\supnorm{\vec{\alpha}})$, and so
	\begin{align}
	\expSumSBoxAt{\tfrac{\vec{a}}{q}+\vec{\alpha}}
	&
	=
	\sum_{1 \leq y_1,\dotsc,y_n \leq q} e\brax[\big]{\aOverQDotLittleF(\vec{y})} 
	\sum_{\substack{\vec{x}\in\bbZ^n\\ \vec{x}/P\in\weightbox \\ \vec{x} \equiv \vec{y}\,\text{mod}\,q}} e(\alphaDotLeadingPart(\vec{x}))
	\nonumber
	\\
	&
	\hphantom{={}}
	+O\brax{P^{n+d-1}\supnorm{\vec{\alpha}}}.
	\label{1.eqn:S_near_a_rational}
	\end{align}
	If $\psi$ is any differentiable complex-valued function on $\bbR^n$, then we have
	\begin{equation*}
	\psi(\vec{x})
	=
	q^{-n}\int_{\substack{\vec{u}\in\bbR^n\\ \supnorm{\vec{u}}\leq q/2}}
	\psi(\vec{x}+\vec{u})
	\,\mathrm{d}\vec{u}
	+O_n\brax[\Big]{ q  \max_{\substack{\vec{u}\in\bbR^n\\ \supnorm{\vec{u}}\leq q/2}} \supnorm{\grad_{\vec{u}} \psi(\vec{x}+\vec{u})}  }.
	\end{equation*}
	Setting $\psi(\vec{x}) =  e(\alphaDotLeadingPart(\vec{x})) $, we deduce that
	\begin{align*}
	\sum_{\substack{\vec{x}\in\bbZ^n\\ \vec{x}/P\in\weightbox \\ \vec{x} \equiv \vec{y}\,\text{mod}\,q}} e(\alphaDotLeadingPart(\vec{x}))
	&=
	q^{-n}\int_{\substack{\vec{v}\in\bbR^n\\ \vec{v}/P \in\weightbox}}
	e(\alphaDotLeadingPart(\vec{v}))
	\,\mathrm{d}\vec{v}
	\\
	&\hphantom{{}={}}
	+O(q^{1-n}P^{n+d-1}\supnorm{\vec{\alpha}}+q^{1-n}P^{n-1}),
	\end{align*}
	where the term $q^{1-n}P^{n-1}$ allows for errors in approximating the boundary of the box $\weightbox$. Substituting into \eqref{1.eqn:S_near_a_rational} shows that
	\begin{equation*}
	\expSumSBoxAt{\tfrac{\vec{a}}{q}+\vec{\alpha}}
	=
	\Sloc{q}{\vec{a}}
	\int_{\substack{\vec{v}\in\bbR^n\\ \vec{v}/P \in\weightbox}}
	e(\alphaDotLeadingPart(\vec{v}))
	\,\mathrm{d}\vec{v}
	+O(qP^{n-1}(1+P^d\supnorm{\vec{\alpha}})).
	\end{equation*}
	To complete the proof of \eqref{1.eqn:Sloc} it suffices to set $\vec{u}=P\vec{t}$ and use the definition of $\Sloc{\infty}{\vec{\gamma}}$ from \S\ref*{1.sec:circle_method_notation}. Now \eqref{1.eqn:integral_over_frakM} follows from \eqref{1.eqn:Sloc} by the definition \eqref{1.eqn:def_of_major_arcs} of \tmajorarcs.
\withqed\end{proof}

We remark that in the case when $\vec{a}=\vec{0}$ and $q=1$, the proof of \eqref{1.eqn:Sloc} is valid whether or not the polynomials $f_i$ have integer coefficients. That is, we always have
\begin{equation}\label{1.eqn:Sinfty}
\expSumSBox
=
P^n\Sloc{\infty}{P^d\vec{\alpha}}
+
\Ocv{}{\vec{f}}\brax[]{  P^{n-1}\brax{ 1+P^{d}\supnorm{\vec{\alpha}} }}
\end{equation}
for any $f_i$ with real cofficients. Next we treat the quantity $\singSeriesIncomplete{P}$ from \eqref{1.eqn:integral_over_frakM}.

\begin{lemma}\label{1.lem:frakS}
	Let the polynomials $f_i$ have integer coefficients, let the box $\weightbox$ from \S\ref*{1.sec:notation} be $\clsd{0}{1}^n$, and let  $\Sloc{q}{\vec{a}}$  be as in \S\ref*{1.sec:circle_method_notation}.
	Suppose we are given $\epsilon \geq 0$ and $C\geq 1$, such that for all $\vec{\alpha},\vec{\beta}\in \bbR^R$ and all $P\geq 1$ the bound \eqref{1.eqn:assumed_arcs} holds. Then:
	\begin{enumerate}[(i)]
		\item\label{1.itm:local_moats}
		There is $\epsilon'>0$ such that $\epsilon'  = O_{\cancellation}(\epsilon)$ and
		\begin{equation}
		\label{1.eqn:local_moats}
		\min\setbrax[\big]{\abs{\Sloc{q}{\vec{a}}},\,\abs{\Sloc{q'}{\vec{a}'}}}
		\llcv{C}{C,\vec{f}}
		(q'+q)^\epsilon
		\supnormbig{\tfrac{\vec{a}}{q}-\tfrac{\vec{a}'}{q'}}^{\frac{\cancellation- \epsilon'}{d-1} }
		\end{equation}
		for all $\vec{a}\in \set{1,\dotsc,q}^R$ and $\vec{a}'\in \set{1,\dotsc,q}^R$ such that $\frac{\vec{a}'}{q'}\neq \frac{\vec{a}}{q}$.
		
		\item\label{1.itm:measure_of_localarcs}
		If $\cancellation>\epsilon'$, then for all $\tee>0$ and $q_0\in\bbN$ we have
		\[
		\#\localarcs{q_0}{\tee}
		\llcv{C}{C,\vec{f}}
		(q_0^\epsilon
		\tee)^{-\frac{(d-1)R}{\cancellation-\epsilon'}
		},
		\]
		where it is understood that the fractions $\tfrac{\vec{a}}{q}$ are in lowest terms.
		
		\item\label{1.itm:bound_on_Sq}
		Let  $\delta_0$  be as in \S\ref*{1.sec:circle_method_notation} and let $\epsilon''>0$. For all $q \in \bbN$ and all $\vec{a} \in \bbZ^R$ such that $(a_1,\dotsc,a_R,q)=1$, we have
		\[
		\abs{\Sloc{q}{\vec{a}}}
		\llcv{\epsilon''}{\vec{f},\epsilon''}
		q^{-\delta_0+\epsilon''}.
		\]
		
		\item\label{1.itm:frakS_converges}
		Let  $\Delta$ and $\singSeriesIncomplete{P}$ be as in \S\ref*{1.sec:circle_method_notation}. Suppose that $\epsilon$ is  sufficiently small in terms of $\cancellation$, $d$ and $R$. Provided the inequality $\cancellation>(d-1)R$ holds and the forms $f^{[d]}_i$ are linearly independent, we have
		\begin{equation}\label{1.eqn:frakS_converges}
		\singSeriesIncomplete{P}-\singSeries
		\llcv{C,\cancellation}{C,\cancellation,\vec{f}}
		P^{-\Delta \delta_1}
		\end{equation}
		for some $\singSeries\in\bbC$ and some $\delta_1 >0$ depending at most on $\cancellation$, $d$ and $R$. We have
		\begin{multline}
		\label{1.eqn:evaluating_frakS}
		\singSeries
		= \prod_p
		\lim_{k\to\infty}
		\tfrac{1}{p^{k(n-R)}}
		\#
		\big\{ \vec{b} \in \set{1,2,\dotsc,p^k}^n
		\suchthat
		\\
		f_1(\vec{b})
		\equiv 0 ,
		\dotsc,
		f_R(\vec{b})
		\equiv
		0
		\mod{p^k}
		\big\}
		\end{multline}
		where the product is over primes $p$ and  converges absolutely.
	\end{enumerate}
\end{lemma}

\begin{proof}[Proof of part~\ref*{1.itm:local_moats}]
	Provided $P$ is sufficiently large, Lemma~\ref*{1.lem:Sloc} will allow us to approximate the sum $\Sloc{q}{\vec{a}}$ by a multiple of $S\brax[\big]{{\vec{a}/q} ;P}$. This will enable us to transform   \eqref{1.eqn:assumed_arcs} into the bound \eqref{1.eqn:local_moats}. Let $P\geq 1$ be a parameter, to be chosen later. Then \eqref{1.eqn:assumed_arcs} gives
	\begin{equation}\label{1.eqn:local_moats_first_step}
	\min\setbrax*{
		\abs*{ \frac{S\brax[\big]{\frac{\vec{a}}{q} ;P}  }{P^{n+\epsilon}} }
		,\,
		\abs*{ \frac{S \brax[\big]{\frac{\vec{a}'}{q'};P} }{P^{n+\epsilon}} }
	}
	\leq
	C
	\max\setbrax{P^{-d}\supnormbig{  \tfrac{\vec{a}'}{q'}-\tfrac{\vec{a}}{q}  }^{-1} ,\, \supnormbig{   \tfrac{\vec{a}'}{q'}-\tfrac{\vec{a}}{q}  }^{\frac{1}{d-1}} }^{\cancellation}.
	\end{equation}
	Since $\weightbox = \clsd{0}{1}^n$ the equality $\Sloc{\infty}{\vec{0}}=1$ holds, and so \eqref{1.eqn:Sloc} implies that
	\begin{equation}\label{1.eqn:Sloc_for_local_moats}
	\frac{ S\brax[\big]{\frac{\vec{a}}{q};P} }{P^{n}}
	=
	\Sloc{q}{ \vec{a} }
	+
	\Ocv{}{\vec{f}}\brax{qP^{-1}},
	\quad
	\frac{ S\brax[\big]{\frac{\vec{a}'}{q'};P} }{P^{n}}
	=
	\Sloc{q'}{ \vec{a}' }
	+
	\Ocv{}{\vec{f}}\brax{q'P^{-1}}.
	\end{equation}
	Together \eqref{1.eqn:local_moats_first_step} and \eqref{1.eqn:Sloc_for_local_moats} yield
	\begin{multline}\label{1.eqn:local_moats_second_step}
	\min\setbrax[\big]{
		\abs{ \Sloc{q}{ \vec{a} }  }
		,\,
		\abs{ \Sloc{q'}{ \vec{a}' } }
	}
	\\ 
	\leq
	C
	P^{\epsilon-\cancellation d}\supnormbig{  \tfrac{\vec{a}'}{q'}-\tfrac{\vec{a}}{q}  }^{-\cancellation} +CP^\epsilon \supnormbig{   \tfrac{\vec{a}'}{q'}-\tfrac{\vec{a}}{q}  }^{\frac{\cancellation}{d-1}}
	+
	\Ocv{}{\vec{f}}\brax[]{ (q'+q)P^{-1}}.
	\end{multline}
	Observe that for $P$ sufficiently large the term $C P^\epsilon  \supnorm{\frac{\vec{a}'}{q'}-\frac{\vec{a}}{q}}^{\cancellation/(d-1)}$ dominates the right-hand side of \eqref{1.eqn:local_moats_second_step}. We claim this is the case for
	\begin{equation}\label{1.eqn:choice_of_P_for_local_moats}
	P
	=
	(q'+q)\supnormbig{   \tfrac{\vec{a}'}{q'}-\tfrac{\vec{a}}{q}  }^{-\frac{1+\cancellation}{d-1}}.
	\end{equation}
	Indeed, since $\supnorm{\frac{\vec{a}'}{q'}-\frac{\vec{a}}{q}}\leq 1$, it follows from \eqref{1.eqn:choice_of_P_for_local_moats} and \eqref{1.eqn:local_moats_second_step} that
	\begin{align*}
	\MoveEqLeft[1]
	\min\setbrax[\big]{
		\abs{ \Sloc{q}{ \vec{a} }  }
		,\,
		\abs{ \Sloc{q'}{ \vec{a}' } }
	}
	\\ 
	&\leq
	C
	P^{\epsilon}
	(q'+q)^{-\cancellation d} \supnormbig{  \tfrac{\vec{a}'}{q'}-\tfrac{\vec{a}}{q}  }^{\frac{\cancellation+\cancellation^2 d}{d-1}} + 	C
	P^{\epsilon}\supnormbig{   \tfrac{\vec{a}'}{q'}-\tfrac{\vec{a}}{q}  }^{\frac{\cancellation}{d-1}} 
	+
	\Ocv{}{\vec{f}}\brax[\Big]{\supnormbig{  \tfrac{\vec{a}'}{q'}-\tfrac{\vec{a}}{q}  }^{\frac{1+\cancellation}{d-1}}}.
	\\
	&\llcv{C}{C,\vec{f}}
	P^\epsilon
	\supnormbig{  \tfrac{\vec{a}'}{q'}-\tfrac{\vec{a}}{q}  }^{\frac{\cancellation}{d-1}},
	\end{align*}
	which proves the result.
\end{proof}
\begin{proof}[Proof of part~\ref*{1.itm:measure_of_localarcs}]
	If $\epsilon'< \cancellation$ is small, then by part~\ref*{1.itm:local_moats}, the points in the set
	\[
	\localarcs{q_0}{\tee}
	\]
	are separated by gaps of size
	\[
	\supnorm{\tfrac{\vec{a}'}{q'}-\tfrac{\vec{a}}{q}}
	\ggcv{C}{C,\vec{f}} (q_0^{-\epsilon} \tee)^{\frac{d-1}{\cancellation-\epsilon'}}.
	\]
	At most $\Ocv{C}{C,\vec{f}}(q_0^{\epsilon} \tee)^{-\frac{(d-1)R}{\cancellation-\epsilon'}}$ such points fit in the box $\clsdopen{0}{1}^R$, proving the claim.
	\end{proof}
\begin{proof}[Proof of part~\ref*{1.itm:bound_on_Sq}]
	This follows from Lemma~\ref*{1.lem:weyl's_ineq} by an argument which is now standard, see the proof of Lemma~5.4 in Birch~\cite{birchFormsManyVars}.
\end{proof}
\begin{proof}[Proof of part~\ref*{1.itm:frakS_converges}]
	In this part of the proof, whenever we write $\vec{a}/q$ it is understood that $\vec{a}\in\bbZ^R$ and $q\in\bbN$ with $(a_1,\dotsc ,a_R,q) =1$. We will show below that
	\begin{equation}\label{1.eqn:tail_of_frakS}
	s(Q)
	=
	\sum_{\substack{\vec{a}/q \in\clsdopen{0}{1}^R \\ Q < q \leq 2Q}}
	\abs{\Sloc{q}{\vec{a}}}
	\llcv{C,\cancellation}{C,\cancellation,\vec{f}}
	Q^{-\delta_1}
	\end{equation}
	for all $Q \geq 1$, and some $\delta_1>0$ depending only on $\cancellation$, $d$ and $R$. Since
	\begin{align*}
	\abs[\bigg]{\singSeriesIncomplete{P}
		-\sum_{ \vec{a}/q\in\clsdopen{0}{1}^R }
		\Sloc{q}{\vec{a}}}
	&\leq
	\sum_{\substack{\vec{a}/q \in\clsdopen{0}{1}^R \\ q> P^\Delta}}
	\abs{\Sloc{q}{\vec{a}}}
	\\
	&=
	\sum_{\substack{ Q = 2^k P^\Delta \\  k = 0, 1, \dotsc}} s(Q),
	\end{align*}
	this proves \eqref{1.eqn:frakS_converges} with
	\begin{equation}
	\label{1.eqn:def_of_frakS}
	\singSeries
	=
	\sum_{ \vec{a}/q\in\clsdopen{0}{1}^R }
	\Sloc{q}{\vec{a}},
	\end{equation}
	where this sum is absolutely convergent. Then  \eqref{1.eqn:evaluating_frakS} follows as in \S 7 of Birch~\cite{birchFormsManyVars}.
	
	We prove \eqref{1.eqn:tail_of_frakS}. Let $\ell \in \bbZ$. We have
	\begin{align}
	s(Q)
	={}&
	\sum_{\substack{  \vec{a}/q\in \clsdopen{0}{1}^R \\ \abs{\Sloc{q}{\vec{a}}}\geq  2^{-\ell} \\  Q < q \leq 2Q }}
	\abs{\Sloc{q}{\vec{a}}}
	+
	\sum_{i = \ell}^\infty
	\sum_{\substack{  \vec{a}/q\in \clsdopen{0}{1}^R \\ 2^{-i}>  \abs{\Sloc{q}{\vec{a}}}\geq 2^{-i-1} \\  Q < q \leq 2Q }}
	\abs{\Sloc{q}{\vec{a}}}
	\nonumber
	\\
	\leq{}&
	\#\localarcs{2Q}{2^{-\ell}} \cdot \sup_{q>Q} \abs{\Sloc{q}{\vec{a}}}
	\nonumber
	\\
	&+
	\sum_{i=\ell}^\infty
	\#\localarcs{2Q}{2^{-i-1}}\cdot 2^{-i}.
	\label{1.eqn:bounding_S(Q)}
	\end{align}
	Now parts~\ref*{1.itm:measure_of_localarcs} and~\ref*{1.itm:bound_on_Sq} show that
	\[
	\#\localarcs{2Q}{\tee} \llcv{C}{C,\vec{f}} (Q^{\epsilon} t)^{-\frac{(d-1)R}{ \cancellation-\epsilon'}}
	\]
	and that
	\[
	\sup_{q>Q} \abs{\Sloc{q}{\vec{a}}} \llcv{}{\vec{f}} Q^{-\delta_0/2}.
	\]
	Substituting these bounds into \eqref{1.eqn:bounding_S(Q)} gives
	\begin{equation*}
	s(Q)
	\llcv{C}{C,\vec{f}}
	Q^{\Ocv{\cancellation}{\cancellation,d,R}(\epsilon)-\delta_0/2}
	2^{\ell \frac{(d-1)R}{\cancellation-\epsilon'}}
	+
	Q^{\Ocv{\cancellation}{\cancellation,d,R}(\epsilon)}
	\sum_{i=\ell}^\infty
	2^{(i+1) \brax[\big]{\frac{(d-1)R}{\cancellation-\epsilon'}}-i}.
	\end{equation*}
	We have $\cancellation>(d-1)R$ and we have assumed that $\epsilon'$ is small in terms of $\cancellation$, $d$ and $R$, so we may assume that the bound $\cancellation>(d-1)R+\epsilon'$ holds. So we may sum the geometric progression to find that 
	\begin{equation*}
	s(Q)	\ll_{C,\cancellation}
	Q^{\Ocv{\cancellation}{\cancellation,d,R}(\epsilon)}
	2^{\ell \frac{(d-1)R}{\cancellation-\epsilon'}}
	\brax[\big]{ 
		Q^{-\delta_0/2} +2^{-\ell} }.
	\end{equation*}
	Picking $
	\ell
	=
	\floor{\log_2 Q^{\delta_0/2}}
	$
	shows that
	\begin{equation*}
	s(Q)
	\llcv{C,\cancellation}{C,\cancellation,d,R}
	Q^{-\delta_0\frac{(d-1)R-\cancellation}{2\cancellation}+\Ocv{\cancellation}{\cancellation,d,R}(\epsilon)}.
	\end{equation*}
	The forms $f^{[d]}_i$ are linearly independent, so $\delta_0\geq \frac{1}{(d-1)2^{d-1}R}$, by Lemma~\ref*{1.lem:Sloc}. As $\epsilon$ is small in terms of $\cancellation$, $d$ and $R$ it follows that $s(Q)\llcv{C,\cancellation}{C,\cancellation,d,R} Q^{-\delta_1}$ for some $\delta_1>0$ depending only on $\cancellation$, $d$ and $R$. This proves \eqref{1.eqn:tail_of_frakS}.
\withqed\end{proof}

We estimate the integral $\singIntegralBoxIncomplete{P}$ from \eqref{1.eqn:integral_over_frakM}.

\begin{lemma}\label{1.lem:frakI}
	Let $\Sloc{\infty}{\vec{\gamma}}$, $\Delta$ and $\singIntegralBoxIncomplete{P}$ be as in \S\ref*{1.sec:circle_method_notation}.
	\begin{enumerate}[(i)]
		\item\label{1.itm:bound_on_Sinfty}
		Suppose that the bound \eqref{1.eqn:assumed_arcs} holds for some $C\geq 1$, $\cancellation>0$ and $\epsilon\geq 0$ and all $\vec{\alpha},\vec{\beta}\in\bbR^R$ and $P\geq 1$. Then for all $\vec{\gamma}\in\bbR^R$ we have
		\begin{equation}\label{1.eqn:bound_on_Sinfty}
		\Sloc{\infty}{\vec{\gamma}}
		\llcv{C}{C,\vec{f}}
		\supnorm{\vec{\gamma}}^{-\cancellation+\epsilon'},
		\end{equation}
		for some $\epsilon'>0$ such that $\epsilon' = O_{\cancellation}(\epsilon)$.
		\item\label{1.itm:frakI_converges}
		If the conclusion of part~\ref*{1.itm:bound_on_Sinfty} holds and $\cancellation-\epsilon'>R$, then there exists $\singIntegralBox \in \bbC$ such that for all $P\geq 1$ we have
		\begin{equation}\label{1.eqn:frakI_converges}
		\tfrac{1}{P^{n-dR}}\singIntegralBoxIncomplete{P}
		-
		\singIntegralBox
		\llcv{\cancellation,C,\epsilon'}{\cancellation,C,\vec{f},\epsilon'}
		P^{-\Delta\brax{\cancellation-\epsilon'-R}}.
		\end{equation}
		Furthermore we have
		\begin{equation}\label{1.eqn:evaluating_frakI}
		\singIntegralBox
		=
		\lim_{P \to \infty}\tfrac{1}{P^{n-dR}}
		\Meas{ \vec{\tee}\in\bbR^n \suchthat \tfrac{1}{P}\vec{\tee} \in \weightbox, \abs{f^{[d]}_1(\vec{\tee})} \leq \tfrac{1}{2} ,\dotsc,
			\abs{f^{[d]}_R(\vec{\tee})} \leq \tfrac{1}{2} 
		}
		\end{equation}
		where $\meas{\,\cdot\,}$ denotes the Lebesgue measure.
	\end{enumerate}
\end{lemma}

\begin{proof}[Proof of part ~\ref*{1.itm:bound_on_Sinfty}]
	First, for all $\vec{\beta}\in\bbR^R$ we have $\abs{\expSumSBoxAt{\vec{\beta}}}\leq \expSumSBoxAt{\vec{0}}$, from the definition \eqref{1.eqn:def_of_S}. Consequently, taking $\vec{\alpha} = \vec{0}$, $\vec{\beta}= P^{-d}\vec{\gamma} $ in our hypothesis \eqref{1.eqn:assumed_arcs} shows that
	\begin{equation*}
	\abs{ \expSumSBoxAt{P^{-d}\vec{\gamma}}}
	\leq
	C
	P^{n+\epsilon}
	\max\setbrax{\supnorm{\vec{\gamma}}^{-1} ,\, P^{-\frac{d}{d-1}}\supnorm{\vec{\gamma}}^{\frac{1}{d-1}} }^{\cancellation}.
	\end{equation*}
	Together with the case $\vec{\alpha}= P^{-d}\vec{\gamma} $ of the bound \eqref{1.eqn:Sinfty}, this yields
	\begin{equation}\label{1.eqn:bound_on_Sinfty_II}
	\Sloc{\infty}{\vec{\gamma}}
	\llcv{C}{C,\vec{f}}
	P^{\epsilon} \max\setbrax{\supnorm{\vec{\gamma}}^{-1},\,
		P^{-\frac{d}{d-1} } \supnorm{\vec{\gamma}}^{\frac{1}{d-1}}}^\cancellation
	+
	P^{-1}
	+
	P^{-1}\supnorm{\vec{\gamma}}.
	\end{equation}
	If  we have $\supnorm{\vec{\gamma}} \leq 1$, then we set $P=1$ and~\ref*{1.itm:bound_on_Sinfty} follows at once. Otherwise we put $P = \max\setbrax{1, \supnorm{\vec{\gamma}}^{1+\cancellation}}$, and the result follows since \eqref{1.eqn:bound_on_Sinfty_II} then implies
	\begin{align*}
	\Sloc{\infty}{\vec{\gamma}}
	&\llcv{C}{C,\vec{f}}
	P^{\epsilon} \max\setbrax[\big]{\supnorm{\vec{\gamma}}^{-1},\,
		\supnorm{\vec{\gamma}}^{-1-\frac{\cancellation d}{d-1}}}^\cancellation
	+\supnorm{\vec{\gamma}}^{-1-\cancellation}
	+\supnorm{\vec{\gamma}}^{-\cancellation}
	\nonumber
	\\
	&\leq
	3 \supnorm{\vec{\gamma}}^{-\cancellation+(1+\cancellation)\epsilon}.
	\end{align*}
\end{proof}
\begin{proof}[Proof of part~\ref*{1.itm:frakI_converges}]	
	If the inequality  $\cancellation-\epsilon'>R $ holds, then by \eqref{1.eqn:bound_on_Sinfty} we have
	\begin{align*}
	\brax[\bigg]{\int\limits_{\vec{\gamma}\in\bbR^R } P^{n-dR} \Sloc{\infty}{\vec{\gamma}}  \,\mathrm{d}\vec{\gamma}}
	-
	\singIntegralBoxIncomplete{P}
	&=
	\int\limits_{\substack{\vec{\gamma}\in\bbR^R \\ \supnorm{\vec{\gamma}}> P^{\Delta} }} P^{n-dR} \Sloc{\infty}{\vec{\gamma}}  \,\mathrm{d}\vec{\gamma}
	\\
	&\llcv{\cancellation,C,\epsilon'}{\cancellation,C,\vec{f},\epsilon'} P^{n-dR-\Delta\brax{\cancellation-\epsilon'-R}} ,
	\end{align*}
	where the integrals converge absolutely. 
	This proves \eqref{1.eqn:frakI_converges} with
	\begin{equation}\label{1.eqn:def_of_frakI}
	\singIntegralBox
	=
	\int\limits_{\vec{\gamma}\in\bbR^R } \Sloc{\infty}{\vec{\gamma}}  \,\mathrm{d}\vec{\gamma}.
	\end{equation}
	It remains to prove \eqref{1.eqn:evaluating_frakI}. Let $\chi : \bbR^R \to \clsd{0}{1}$ be the indicator function of the box $\clsd{-\frac{1}{2}}{\frac{1}{2}}^R$. We must evaluate the limit
	\begin{multline}\label{1.eqn:sing_int_as_volume_integral}
	\lim_{P \to \infty}
	\tfrac{1}{P^{n-dR}}
	\Meas{ \vec{\tee}\in\bbR^n \suchthat \tfrac{1}{P}\vec{\tee} \in \weightbox, \abs{f^{[d]}_1(\vec{\tee})} \leq \tfrac{1}{2} ,\dotsc,
		\abs{f^{[d]}_R(\vec{\tee})} \leq \tfrac{1}{2} }
	\\
	=
	\lim_{P\to\infty}
	\tfrac{1}{P^{n-dR}}
	\int_{\substack{ \vec{\tee}\in\bbR^n \\ \vec{\tee}/P \in\weightbox }}
	\chi\brax[\big]{ f^{[d]}_1(\vec{\tee}) ,\dotsc,
		f^{[d]}_R(\vec{\tee})  } \,\mathrm{d}\vec{\tee}.
	\end{multline}
	Let  $\varphi$ be any infinitely differentiable, compactly supported function on $\bbR^R$, taking values in $\clsd{0}{1}$. We evaluate $\frac{1}{P^{n-dR}}\int_{\vec{t}/P\in\weightbox} \varphi(  f^{[d]}_1(\vec{\tee}) ,\dotsc,
	f^{[d]}_R(\vec{\tee})  )\,\mathrm{d}\vec{t}$,  which we think of as a smoothed version of \eqref{1.eqn:sing_int_as_volume_integral}. Fourier inversion gives
	\begin{align}
	\int_{\substack{ \vec{\tee}\in\bbR^n \\ \vec{\tee}/P \in\weightbox }}
	\varphi\brax[\big]{ f^{[d]}_1(\vec{\tee}) ,\dotsc,
		f^{[d]}_R(\vec{\tee}) } \,\mathrm{d}\vec{\tee}
	&=
	\int_{\substack{ \vec{\tee}\in\bbR^n \\ \vec{\tee}/P \in\weightbox }}
	\int_{\bbR^R}
	\hat{\varphi}(\vec{\alpha})
	e\brax{ \alphaDotLeadingPart(\vec{\tee}) } \,\mathrm{d}\vec{\alpha} d\vec{\tee}
	\nonumber
	\\
	&=
	\int_{\bbR^R}
	\hat{\varphi}(\vec{\alpha})
	\int_{\substack{ \vec{\tee}\in\bbR^n \\ \vec{\tee}/P \in\weightbox }}
	e\brax{ \alphaDotLeadingPart(\vec{\tee}) } \,\mathrm{d}\vec{\tee} d\vec{\alpha}
	\nonumber
	\\
	&=
	\int_{\bbR^R}
	\hat{\varphi}(\vec{\alpha})
	P^n \Sloc{\infty}{P^d\vec{\alpha}} d\vec{\alpha}
	\label{1.eqn:fourier_transform_frakI-sub-varphi}
	\end{align}
	where $\hat{\varphi}(\vec{\alpha})$ is the Fourier transform $ \int_{\bbR^R} \varphi(\vec{\gamma}) e\brax{- \vec{\alpha}\cdot\vec{\gamma} } \,\mathrm{d}\vec{\gamma}$. 
	
	Since $ \cancellation-\epsilon' >R$ holds by assumption, it follows from \eqref{1.eqn:bound_on_Sinfty} that the function $S_\infty$ is Lebesgue integrable. Hence \eqref{1.eqn:def_of_frakI} implies
	\begin{align}
	\hat{\varphi}(\vec{0}) \singIntegralBox
	\nonumber
	&=
	\int_{\bbR^R}
	\hat{\varphi}(\vec{0})
	\Sloc{\infty}{ \vec{\gamma}} \,\mathrm{d}\vec{\gamma}
	\nonumber
	\\
	&=
	\lim_{P\to\infty}
	\int_{\bbR^R}
	\hat{\varphi}(P^{-d}\vec{\gamma})
	\Sloc{\infty}{ \vec{\gamma}} \,\mathrm{d}\vec{\gamma}
	\nonumber
	\\
	&=
	\lim_{P\to\infty}
	P^{dR}
	\int_{\bbR^R}
	\hat{\varphi}(\vec{\alpha})
	\Sloc{\infty}{P^d \vec{\alpha}} \,\mathrm{d}\vec{\alpha}.
	\label{1.eqn:dominated_convergence_trick}
	\end{align}
	Together \eqref{1.eqn:fourier_transform_frakI-sub-varphi} and \eqref{1.eqn:dominated_convergence_trick} show that for any infinitely differentiable, compactly supported $\varphi$ taking values in $\clsd{0}{1}$, we have
	\begin{equation}\label{1.eqn:evaluation_of_frakI-sub-varphi}
	\lim_{P\to\infty}
	\tfrac{1}{P^{n-dR}}
	\int_{\substack{ \vec{\tee}\in\bbR^n \\ \vec{\tee}/P \in\weightbox }}
	\varphi\brax{  f^{[d]}_1(\vec{\tee}) ,\dotsc,
		f^{[d]}_R(\vec{\tee}) } \,\mathrm{d}\vec{\tee}
	=
	\hat{\varphi}(\vec{0}) \singIntegralBox.
	\end{equation}
	With $\chi$ as in \eqref{1.eqn:sing_int_as_volume_integral}, choose $\varphi$ such that $\varphi(\vec{\gamma})\leq \chi(\vec{\gamma})$ for all $\vec{\gamma}\in \bbR^R$. Then by \eqref{1.eqn:sing_int_as_volume_integral} and \eqref{1.eqn:evaluation_of_frakI-sub-varphi} we have
	\begin{equation*}
	\liminf_{P \to \infty}\tfrac{1}{P^{n-dR}}
	\Meas{ \vec{\tee}\in\bbR^n \suchthat \tfrac{1}{P}\vec{\tee} \in \weightbox, \abs{f^{[d]}_1(\vec{\tee})} \leq \tfrac{1}{2} ,\dotsc,
		\abs{f^{[d]}_R(\vec{\tee})} \leq \tfrac{1}{2}  }
	\geq
	\hat{\varphi}(\vec{0})
	\singIntegralBox.
	\end{equation*}
	Letting $\varphi \to \chi$ almost everywhere gives $
	\hat{\varphi}(\vec{0}) \to 1$, so $\singIntegralBox$ is a lower bound for the limit inferior in \eqref{1.eqn:sing_int_as_volume_integral}. Repeating the argument with $\varphi(\vec{\gamma})\geq \chi(\vec{\gamma})$ instead of $\varphi(\vec{\gamma})\leq \chi(\vec{\gamma})$ shows that $\singIntegralBox$ is also an upper bound for the corresponding limit superior, so the limit exists and is equal to $\singIntegralBox$.
\withqed\end{proof}

\subsection{The proof of Proposition~\ref*{1.prop:circle_method}}\label{1.sec:completing_the_circle_method}

In this section we deduce Proposition~\ref*{1.prop:circle_method} from Lemmas~\ref*{1.lem:mean_value_from_power-law_moats},~\ref*{1.lem:weyl's_ineq},~\ref*{1.lem:Sloc},~\ref*{1.lem:frakS}, and ~\ref*{1.lem:frakI}.

\begin{proof}[Proof of Proposition~\ref*{1.prop:circle_method}]
	Let $P\geq 1$ and $\Delta =\frac{1}{4R+6}$. By \eqref{1.eqn:circle_method} we have
	\begin{equation*}
	\numSolnsInBox(P)
	=
	\int_{\majorarcs} \expSumSBox\,\mathrm{d}\vec{\alpha}
	+\int_{\minorarcs} \expSumSBox\,\mathrm{d}\vec{\alpha},
	\end{equation*}
	where $\majorarcs$, $\minorarcs$ are as in \S\ref*{1.sec:circle_method_notation}. We apply Lemma~\ref*{1.lem:mean_value_from_power-law_moats} with
	\[
	T(\vec{\alpha}) = C^{-1}P^{-\epsilon}\expSumSBox
	,\quad
	E_0 = \clsd{0}{1}^R,\quad
	E = \minorarcs,\quad
	\delta=\Delta\delta_0.
	\]
	With these choices for $T$, $E_0$, $E$ and $\delta$ we see that \eqref{1.eqn:power-law_moats} follows from \eqref{1.eqn:assumed_arcs}. Lemma~\ref*{1.lem:weyl's_ineq} shows that $\sup_{\vec{\alpha}\in\minorarcs}CT(\vec{\alpha}) \ll_\epsilon P^{n-\delta}$, and after increasing $C$ if necessary this gives us \eqref{1.eqn:weyl's_ineq_mean_value_lem}. This verifies the hypotheses of Lemma~\ref*{1.lem:mean_value_from_power-law_moats}. Since we have $\cancellation>dR$ by assumption, \eqref{1.eqn:mean_value_lem} gives
	\begin{equation}\label{1.eqn:mean_value_on_minor_arcs}
	\int_{\minorarcs}\expSumSBox\,\mathrm{d}\vec{\alpha}
	\llcv{C,\cancellation}{C,\cancellation,d,R}
	P^{n-dR-\Delta\delta_0\brax{1-\frac{dR}{\cancellation}}+\epsilon}.
	\end{equation}
	For the major arcs, since $\Delta=\frac{1}{4R+6}$ we have by Lemma~\ref*{1.lem:Sloc} that 
	\begin{equation}\label{1.eqn:integral_over_frakM_application}
	\int_{\majorarcs} \expSumSBox \,\mathrm{d}\vec{\alpha}
	=
	\singSeriesIncomplete{P}\singIntegralBoxIncomplete{P}
	+ \Ocv{}{\vec{f}} \brax[\big]{ P^{n-dR-\frac{1}{2}}},
	\end{equation}
	where  $\singSeriesIncomplete{P}$, $\singIntegralBoxIncomplete{P}$ are as in \S\ref*{1.sec:circle_method_notation}. Since $\cancellation>dR$ holds, the $f_i(\vec{x})$ are linearly independent, and $\epsilon$ is small in terms of $\cancellation$, $d$ and $R$, both of Lemmas~\ref*{1.lem:frakS} and~\ref*{1.lem:frakI} apply. In particular \eqref{1.eqn:frakS_converges} and \eqref{1.eqn:frakI_converges} shows that
	\begin{equation}\label{1.eqn:product_of_densities}
	\singSeriesIncomplete{P}\singIntegralBoxIncomplete{P}
	=
	\singSeries\singIntegralBox P^{n-dR}
	+\Ocv{\cancellation,C}{\cancellation,C,\vec{f}}
	\brax[\big]{
		P^{n-dR-\Delta(\cancellation-R)/2}
	}
	+\Ocv{\cancellation,C}{\cancellation,C,\vec{f}}
	\brax[\big]{
		P^{n-dR-\Delta\delta_1}
	}
	\end{equation}
	where $\delta_1 >0$ depends at most on $\cancellation$, $d$ and $R$. By \eqref{1.eqn:mean_value_on_minor_arcs}, \eqref{1.eqn:integral_over_frakM_application}, and \eqref{1.eqn:product_of_densities}, the result holds.
\withqed\end{proof}

\section{The auxiliary inequality}\label{1.sec:aux_ineq}

In this section we verify the hypothesis \eqref{1.eqn:assumed_arcs}, assuming a bound on the number of solutions to the auxiliary inequality from Definition~\ref*{1.def:aux_ineq}. The goal is the following result, proved at the end of \S\ref*{1.sec:proof_of_moat_lemma}.

\begin{proposition}\label{1.prop:moat_lemma}
	Let $\auxIneqOfSomethingNumSolns{f}\brax{ B }$, $\supnorm{f}$  be as in Definition~\ref*{1.def:aux_ineq}. Suppose that we are given $C_0\geq 1$ and $\cancellation>0$ such that for all $\vec{\beta}\in\bbR^R$ and $B\geq 1$ we have
	\begin{equation}\label{1.eqn:aux_ineq_bound_moat_lemma}
	\auxIneqNumSolns \brax{ B }
	\leq
	C_0
	B^{(d-1)n-2^{d}\cancellation}.
	\end{equation}
	Further let $M > \mu >0$ such that for all $\vec{\beta}\in\bbR^R$ we have
	\begin{equation}\label{1.eqn:norm_of_|b.f|}
	\mu \supnorm{ \vec{\beta} }
	\leq
	\supnorm{ \betaDotLeadingPart }
	\leq
	M \supnorm{ \vec{\beta} },
	\end{equation}
	noting that some such $M, \mu$ exist whenever the forms $f^{[d]}_i$ are linearly independent. Let  $\epsilon>0$. Then there exists $C\geq 1$, depending only on $C_0,d,n,\mu,M$ and $\epsilon$, such that the bound \eqref{1.eqn:assumed_arcs} holds for all $P\geq 1 $ and all $\vec{\alpha},\vec{\beta}\in\bbR^R$.
\end{proposition}

\subsection{Weyl differencing}\label{1.sec:weyl_diff}

We prove \eqref{1.eqn:assumed_arcs} using the following estimate, which combines work of Birch~\cite[Lemma~2.4]{birchFormsManyVars} and Bentkus and G\"otze~\cite[Theorem~5.1]{bentkusGotzeEllipsoids}.

\begin{definition}\label{1.def:weyl_diff_ineq}
	Let $f$, $\gradSomethingMultilinear{f}{\vec{x}^{(1)}, \dotsc, \vec{x}^{(d-1)}}$ be as in  Definition~\ref*{1.def:aux_ineq}. Given  $B \geq 1$ and $\delta>0$, we let  $\weylDiffIneqOfSomethingNumSolns{f} \brax{B,\delta}$ be the number of $(d-1)$-tuples of integer $n$-vectors $\vec{x}^{(1)},\dotsc,\vec{x}^{(d-1)}$ such that
	\begin{equation*}
	\supnorm{\vec{x}^{(1)}},\dotsc,\supnorm{\vec{x}^{(d-1)}} \leq B,
	\qquad
	\min_{\vec{v}\in\bbZ^n} \supnormbig{\vec{v}-\gradSomethingMultilinear{ f }{  \vec{x}^{(1)}, \dotsc, \vec{x}^{(d-1)}}} < \delta.
	\end{equation*}
\end{definition}

\begin{lemma}\label{1.lem:weyl_diff_ineq}
	Let $\weylDiffIneqOfSomethingNumSolns{f} \brax{ B,\delta  }$ be as in Definition~\ref*{1.def:weyl_diff_ineq}. For all $\epsilon>0$, $\vec{\alpha}, \vec{\beta}\in\bbR^R$ and $\theta \in \openclsd{0}{1}$,  we have
	\begin{equation}
	\label{1.eqn:from_min_to_prod}
	\min\setbrax*{ \abs*{\frac{\expSumSBox}{P^{n+\epsilon}}} , \,  \abs*{\frac{\expSumSBoxAt{\vec{\alpha}+\vec{\beta}}}{P^{n+\epsilon}}} }^{2^d}
	\ll_{d,n,\epsilon}
	\frac{
		\weylDiffIneqNumSolns \brax{ P^\theta, P^{(d-1)\theta-d} }
	}{P^{(d-1)\theta n}}
	\end{equation}
	where the implicit constant depends only on $d,n,\epsilon$.
\end{lemma}

\begin{proof}
	Observe that \eqref{1.eqn:from_min_to_prod} will follow if we can prove that
	\begin{equation*}
	\abs*{\frac{\expSumSBox \expSumSBoxAt{\vec{\alpha}+\vec{\beta}}}{P^{2(n+\epsilon)}}}^{2^{d-1}}
	\ll_{d,n,\epsilon}
	\frac{
		\weylDiffIneqNumSolns \brax{ P^\theta, P^{(d-1)\theta-d} }
	}{P^{(d-1)\theta n}}.
	\end{equation*}
	
	First we use an idea from the proof of Theorem~5.1 in Bentkus and G\"otze~\cite{bentkusGotzeEllipsoids}, also found in Lemma~2.2 of M\"uller~\cite{mullerSystemsQuadIneqs}, to eliminate $\vec{\alpha}$.
	We have
	\begin{align*}
	\MoveEqLeft[2]
	\expSumSBoxAt{\vec{\alpha}+\vec{\beta}} \widebar{\expSumSBox}
	\nonumber
	\\
	&=
	\sum_{\substack{\vec{x} \in \bbZ^n \\ \vec{x}/P\in\weightbox }}
	\sum_{\substack{  \vec{z} \in \bbZ^n \\ \brax{\vec{x}+\vec{z}}/P\in\weightbox  }}
	e\brax[\big]{\brax{\vec{\alpha}+\vec{\beta}}\cdot\vec{f}(\vec{x}) -\alphaDotLittleF(\vec{x}+\vec{z})}
	\nonumber
	\\
	&\leq
	\sum_{\substack{  \vec{z} \in \bbZ^n \\ \supnorm{\vec{z}}\leq P  }}
	\abs[\bigg]{
		\sum_{\substack{\vec{x} \in \bbZ^n \\ \vec{x}/P\in\weightbox_{\vec{z}} }}
		e\brax[\big]{\brax{\vec{\alpha}+\vec{\beta}}\cdot\vec{f}(\vec{x}) -\alphaDotLittleF(\vec{x}+\vec{z})}
	}
	\nonumber
	\\
	&=
	\sum_{\substack{  \vec{z} \in \bbZ^n \\ \supnorm{\vec{z}}\leq P  }}
	\abs[\bigg]{
		\sum_{\substack{\vec{x} \in \bbZ^n \\ \vec{x}/P\in\weightbox_{\vec{z}} }}
		e\brax[\big]{\betaDotLeadingPart(\vec{x})  + g_{\vec{\alpha},\vec{\beta},\vec{z}}(\vec{x})
		}
	}
	\end{align*}
	for some real polynomials $g_{\vec{\alpha},\vec{\beta},\vec{z}}(\vec{x})$ of degree at most $d-1$ in $\vec{x}$, and some boxes $\weightbox_{\vec{z}}\subset\weightbox$. Now by the special case	of Cauchy's inequality 
	$
	\abs{\sum_{i\in\calI} \lambda_i}^2
	\leq
	\brax{\#\calI} \cdot \sum_{i\in\calI} \abs{\lambda_i}^2
	$, we have
	\begin{align}
	\MoveEqLeft[1]
	\abs{\expSumSBoxAt{\vec{\alpha}+\vec{\beta}} \widebar{\expSumSBox}}^{2^{d-1}}
	\nonumber
	\\
	&\leq
	\brax[\bigg]{\sum_{\substack{  \vec{z} \in \bbZ^n \\ \supnorm{\vec{z}}\leq P  }}
		\abs[\bigg]{
			\sum_{\substack{\vec{x} \in \bbZ^n \\ \vec{x}/P\in\weightbox_{\vec{z}} }}
			e\brax[\big]{\betaDotLeadingPart(\vec{x})  + g_{\vec{\alpha},\vec{\beta},\vec{z}}(\vec{x}) }
		}
	}^{2^{d-1}}
	\nonumber
	\\
	&\llcv{}{d,n}
	P^{(2^{d-1}-1)n}
	\sum_{\substack{  \vec{z} \in \bbZ^n \\ \supnorm{\vec{z}}\leq P  }}
	\abs[\bigg]{
		\sum_{\substack{\vec{x} \in \bbZ^n \\ \vec{x}/P\in\weightbox_{\vec{z}} }}
		e\brax[\big]{\betaDotLeadingPart(\vec{x})  + g_{\vec{\alpha},\vec{\beta},\vec{z}}(\vec{x}) }
	}^{2^{d-1}}.
	\label{1.eqn:WD2}
	\end{align}
	Bentkus and G\"otze used the double large sieve of Bombieri and Iwaniec \cite{bombieriDoubleSievePaper} to bound the inner sum in \eqref{1.eqn:WD2} in the case when $d=2$. We extend the argument to higher $d$ by employing Lemma~2.4 of Birch~\cite{birchFormsManyVars}, which states that\footnote{Birch writes $N \brax{ P^\theta ; P^{(d-1)\theta-d} ;  {\alpha} }$ for our $\weylDiffIneqOfSomethingNumSolns{\alphaDotLittleF} \brax{ P^\theta, P^{(d-1)\theta-d}  }$ and $S({\alpha})$ for our $\expSumSBox$.}
	\[
	\expSumSBox
	\ll_{d,n,\epsilon}
	P^{2^{d-1}n-(d-1)n\theta +\epsilon}  \weylDiffIneqOfSomethingNumSolns{\alphaDotLittleF} \brax{ P^\theta, P^{(d-1)\theta-d}  }.
	\]
	The innermost sum in \eqref{1.eqn:WD2} has the same form as $\expSumSBox$, with $\weightbox_{\vec{z}}$ in place of $\weightbox$ and $\betaDotLeadingPart(\vec{x})  + g_{\vec{\alpha},\vec{\beta},\vec{z}}(\vec{x}) $ in place of $\alphaDotLittleF$ as the underlying polynomial. The degree of $ g_{\vec{\alpha},\vec{\beta},\vec{z}}$ is at most $d-1$, so $\betaDotLeadingPart(\vec{x})$ is the leading part of this polynomial. So applying Birch's result to the innermost sum in \eqref{1.eqn:WD2} shows
	\begin{align*}
	\MoveEqLeft[6]
	\abs[\bigg]{
		\sum_{\substack{\vec{x} \in \bbZ^n \\ \vec{x}/P\in\weightbox_{\vec{z}} }}
		e\brax[\big]{\betaDotLeadingPart(\vec{x})  + g_{\vec{\alpha},\vec{\beta},\vec{z}}(\vec{x}) }
	}^{2^{d-1}}
	\\
	&	\llcv{\epsilon}{d,n,\epsilon}
	P^{2^{d-1}n-(d-1)\theta n+\epsilon}
	\weylDiffIneqOfSomethingNumSolns{\betaDotLeadingPart(\vec{x})  + g_{\vec{\alpha},\vec{\beta},\vec{z}}(\vec{x})} \brax{ P^\theta, P^{(d-1)\theta-d} }
	\\
	&=
	P^{2^{d-1}n-(d-1)\theta n+\epsilon}
	\weylDiffIneqNumSolns \brax{ P^\theta, P^{(d-1)\theta-d} },
	\end{align*}
	as $\weylDiffIneqOfSomethingNumSolns{f}$ depends only on the degree $d$ part of $f$. With \eqref{1.eqn:WD2} this proves the result.
\withqed\end{proof}

\subsection{Proof of Proposition~\ref*{1.prop:moat_lemma}}\label{1.sec:proof_of_moat_lemma}

\begin{proof}[Proof of Proposition~\ref*{1.prop:moat_lemma}]
	Let us first suppose that for some $\theta >0$ we have
	\begin{equation}\label{1.eqn:M<N}
	\auxIneqNumSolns \brax{ P^\theta }
	<
	\weylDiffIneqNumSolns \brax{ P^\theta , P^{(d-1)\theta-d}  }.
	\end{equation}
	Then there must be a $(d-1)$-tuple of vectors $\vec{x}^{(1)},\dotsc,\vec{x}^{(d-1)} \in \bbZ^n$ which is included in the count $\weylDiffIneqNumSolns \brax{ P^\theta  , P^{(d-1)\theta-d} }$ but not in $\auxIneqNumSolns \brax{P^\theta }$.
	
	Since the $(d-1)$-tuple $(\vec{x}^{(1)},\dotsc,\vec{x}^{(d-1)})$ is counted by $\weylDiffIneqNumSolns \brax{ P^\theta  , P^{(d-1)\theta-d} }$, the inequality $\supnorm{\vec{x}^{(i)}}\leq P^\theta $ holds for each $i=1,\dotsc,d-1$, and we have the bound
	\begin{align}
	\label{1.eqn:||f(xhat)||<stuff}
	\supnormbig{
		\vec{v}-\gradFMultilinear{ \vec{\beta} }{  \vec{x}^{(1)}, \dotsc, \vec{x}^{(d-1)}}
	}
	&<
	P^{(d-1)\theta-d},
	\intertext{for some $\vec{v}\in \bbZ^n$. Since this $(d-1)$-tuple  $(\vec{x}^{(1)},\dotsc,\vec{x}^{(d-1)})$ is not counted by $\auxIneqNumSolns \brax{P^\theta }$, we must also have}
	\label{1.eqn:|f(xhat)|>alpha}
	\supnorm{
		\gradFMultilinear{ \vec{\beta} }{  \vec{x}^{(1)}, \dotsc, \vec{x}^{(d-1)}}
	}
	&\geq
	\supnorm{\betaDotLeadingPart} P^{(d-2)\theta}.
	\end{align}
	We use \eqref{1.eqn:||f(xhat)||<stuff} and \eqref{1.eqn:|f(xhat)|>alpha} to relate $P^\theta$ and $\supnorm{\vec{\beta}}$. It follows from \eqref{1.eqn:||f(xhat)||<stuff} that either
	\begin{gather}
	\label{1.eqn:|f(xhat)|<theta^d-1_/_P}
	\supnormbig{\gradFMultilinear{ \vec{\beta} }{  \vec{x}^{(1)}, \dotsc, \vec{x}^{(d-1)}}} < P^{(d-1)\theta-d}
	\intertext{or}
	\label{1.eqn:|f(xhat)|>1/2}
	\supnormbig{\gradFMultilinear{ \vec{\beta} }{  \vec{x}^{(1)}, \dotsc, \vec{x}^{(d-1)}}} \geq \frac{1}{2}.
	\end{gather}
	When \eqref{1.eqn:|f(xhat)|<theta^d-1_/_P} holds, then \eqref{1.eqn:|f(xhat)|>alpha} implies
	\begin{equation}\label{1.eqn:alpha-alpha'_<_P^-d_(theta_P)^d-1}
	\supnorm{\betaDotLeadingPart}
	<
	\frac{P^{(d-1)\theta-d}}{P^{(d-2)\theta}}
	=
	P^{\theta-d}.
	\end{equation}
	When on the other hand \eqref{1.eqn:|f(xhat)|>1/2} holds, then the bound $\supnorm{\vec{x}^{(i)}}\leq P^\theta $ implies
	\begin{equation*}
	\supnorm{
		\gradFMultilinear{ \vec{\beta} }{  \vec{x}^{(1)}, \dotsc, \vec{x}^{(d-1)}}
	}
	\llcv{}{d,n}
	\supnorm{\betaDotLeadingPart}
	P^{(d-1)\theta},
	\end{equation*}
	and it follows by \eqref{1.eqn:|f(xhat)|>1/2} that
	\begin{equation}\label{1.eqn:alpha-alpha'_>_(theta_P)^1-d}
	\supnorm{\betaDotLeadingPart}
	\ggcv{}{d,n}
	P^{-(d-1)\theta}.
	\end{equation}
	Either \eqref{1.eqn:alpha-alpha'_<_P^-d_(theta_P)^d-1} or \eqref{1.eqn:alpha-alpha'_>_(theta_P)^1-d} holds. So by rearranging and applying \eqref{1.eqn:norm_of_|b.f|} we infer
	\begin{equation}\label{1.eqn:bound_on_theta_P}
	P^{-\theta}
	\ll_{\mu, M}
	\max\setbrax{P^{-d}\supnorm{\vec{\beta}}^{-1},\,\supnorm{\vec{\beta}}^{\frac{1}{d-1}}}.
	\end{equation}
	
	We have shown that \eqref{1.eqn:M<N} implies \eqref{1.eqn:bound_on_theta_P}. Now Lemma~\ref*{1.lem:weyl_diff_ineq}  shows that for $\theta\in\openclsd{0}{1}$  we have
	\begin{equation*}
	\weylDiffIneqNumSolns\brax{ P^\theta, P^{(d-1)\theta-d} }
	\ggcv{\epsilon}{d,n,\epsilon}
	P^{(d-1)\theta n}
	\min\setbrax*{ \abs*{\frac{\expSumSBox}{P^{n+\epsilon}}} , \,  \abs*{\frac{\expSumSBoxAt{\vec{\alpha}+\vec{\beta}}}{P^{n+\epsilon}}} }^{2^d},
	\end{equation*}
	and together with our assumption \eqref{1.eqn:aux_ineq_bound_moat_lemma} this implies that \eqref{1.eqn:M<N} will hold provided that $\theta\in\openclsd{0}{1}$ and that
	\begin{equation}\label{1.eqn:M<N_sufficient_condition}
	(P^\theta)^{(d-1)n - 2^d \cancellation }
	\leq
	C_1^{-1}
	P^{(d-1)\theta n}
	\min\setbrax*{ \abs*{\frac{\expSumSBox}{P^{n+\epsilon}}} , \,  \abs*{\frac{\expSumSBoxAt{\vec{\alpha}+\vec{\beta}}}{P^{n+\epsilon}}} }^{2^d}
	\end{equation}
	for some $C_1 \geq 1$ depending only on $C_0,d,n$ and $\epsilon$. Define $\theta$ by
	\begin{equation}\label{1.eqn:choice_of_theta}
	P^\theta
	=
	C_1^{1/2^d \cancellation}
	\min\setbrax*{ \abs*{\frac{\expSumSBox}{P^{n+\epsilon}}} , \,  \abs*{\frac{\expSumSBoxAt{\vec{\alpha}+\vec{\beta}}}{P^{n+\epsilon}}} }^{-1/\cancellation},
	\end{equation}
	so that equality holds in \eqref{1.eqn:M<N_sufficient_condition}. We consider three cases.
	
	The first case is when $\theta\leq 0$ holds. We can rule this out. If $\theta \leq 0$ then  \eqref{1.eqn:choice_of_theta} gives
	\begin{equation}
	\label{1.eqn:theta<0}
	\min\setbrax*{ \abs*{\frac{\expSumSBox}{P^{n+\epsilon}}} , \,  \abs*{\frac{\expSumSBoxAt{\vec{\alpha}+\vec{\beta}}}{P^{n+\epsilon}}} }
	\geq
	C_1^{-1/2^d}.
	\end{equation}
	To prove \eqref{1.eqn:assumed_arcs}, we can assume without loss of generality that $P \ggcv{\epsilon}{n,\epsilon} 1$ holds. But then \eqref{1.eqn:theta<0} is false, since $\abs{\expSumSBox}\leq (P+1)^n$ by the definition \eqref{1.eqn:def_of_S}.

The second case  is when $0 < \theta\leq 1$ holds. Our choice \eqref{1.eqn:choice_of_theta} for the parameter $\theta$ then ensures that \eqref{1.eqn:M<N_sufficient_condition} holds. We saw above that when $\theta\in\openclsd{0}{1}$, that bound \eqref{1.eqn:M<N_sufficient_condition} implies the inequality \eqref{1.eqn:M<N}. We also saw that \eqref{1.eqn:M<N} leads to the estimate \eqref{1.eqn:bound_on_theta_P}.  This estimate \eqref{1.eqn:bound_on_theta_P} implies the conclusion  \eqref{1.eqn:assumed_arcs} of the lemma upon substituting in the value of~$\theta$ from \eqref{1.eqn:choice_of_theta} and choosing $C$ to satisfy the bound $C \gg_{\mu,M} C_1^{1/2^d}$.

The third and last case is when $\theta> 1$ holds. In this case we have by \eqref{1.eqn:choice_of_theta} that
	\begin{equation}
	\label{1.eqn:theta>1}
	\min\setbrax*{ \abs*{\frac{\expSumSBox}{P^{n+\epsilon}}} , \,  \abs*{\frac{\expSumSBoxAt{\vec{\alpha}+\vec{\beta}}}{P^{n+\epsilon}}} }
	<
	C_1^{1/2^d}
	P^{-\cancellation}.
	\end{equation}
	Now for any $\tee >0$ we have $\max\setbrax{P^{-d} \tee^{-1},\,\tee^{\frac{1}{d-1}}} \geq P^{-1}$, and hence
	\begin{equation*}
	\max\setbrax{P^{-d}\supnorm{\vec{\beta}}^{-1},\,\supnorm{\vec{\beta}}^{\frac{1}{d-1}}}^\cancellation
	\geq P^{-\cancellation}.
	\end{equation*}
	So  \eqref{1.eqn:assumed_arcs} follows from \eqref{1.eqn:theta>1} on choosing $C$ such that $C\geq C_1^{1/2^d}$ holds.\withqed
\end{proof}

\section{The proof of Theorems~\ref*{1.thm:main_thm_short} and~\ref*{1.thm:manin}}\label{1.sec:main_thm_proof}

\begin{proof}[Proof of Theorem~\ref*{1.thm:manin}]
	Let ${\numZeroesInBoxOf{F_1,\dotsc,F_R}}(P)$ be as in \eqref{1.eqn:def_of_num_solns_in_box}. Set $f_i = F_i$, and apply Propositions~\ref*{1.prop:circle_method} and~\ref*{1.prop:moat_lemma}.  This shows that
	\begin{equation}\label{1.eqn:HL_formula_application}
	\numZeroesInBoxOf{f_1,\dotsc,f_R}(P)
	=
	\singSeries\singIntegralBox P^{n-dR}
	+O_{C,f_1,\dotsc,f_R}(P^{n-dR-\delta}),
	\end{equation}
	where  $\delta= \delta(\cancellation,d,R)$ is positive. It remains to prove that $\singIntegralBox$ and $\singSeries$   are positive under the conditions given in the theorem. Note that since $V(F_1,\dotsc,F_R)$ has dimension $n-1-R$, a smooth point corresponds to a solution of the equations
	\begin{equation}
	F_1(\vec{x}) = 0,\dotsc,F_R(\vec{x})=0
	\label{1.eqn:F=0}
	\end{equation}
	at which the $R\times n$ Jacobian matrix $ \brax{\partial F_i(\vec{x})/\partial x_j}_{ij}$ has full rank.
	
	Let $\vec{x}=\vec{r}$ be a real solution to \eqref{1.eqn:F=0} at which the matrix $ \brax{\partial F_i(\vec{x})/\partial x_j}_{ij}$ has full rank, and for which $\vec{r}\in\weightbox$. Applying the Implicit Function Theorem to the equations \eqref{1.eqn:F=0} at the point $\vec{r}$, we find an open set $U\subset \weightbox$ on which the solutions to \eqref{1.eqn:F=0} form an $(n-R)$ dimensional real manifold. Considering a small neighbourhood of this manifold shows that for all $\epsilon\in\openclsd{0}{1}$ we have
	\begin{equation*}
	\Meas{  \vec{s} \in U \suchthat \abs{F_1(\vec{s})} \leq \epsilon, \dotsc,  \abs{F_R(\vec{s})}  \leq \epsilon}
	\gg_{F_1,\dotsc,F_R}
	\epsilon^R
	\end{equation*}
	where $\lambda$ is the Lebesgue measure. Letting $\vec{t} = P\vec{s}$ and $\epsilon = \tfrac{1}{2}P^{-d}$, we see that
	\[
	\Meas{ \vec{t}\in\bbR^n\suchthat \vec{t}/P \in U, \abs{F_1(\vec{t})} \leq \tfrac{1}{2}, \dotsc,  \abs{F_R(\vec{t})}  \leq \tfrac{1}{2}}
	\gg_{F_1,\dotsc,F_R}
	P^{n-dR},
	\]
	and \eqref{1.eqn:evaluating_frakI} from Lemma~\ref*{1.lem:frakI} then shows that $\singIntegralBox$ is positive.
	
	To show that $\singSeries$ is positive under the conditions given in the theorem we use a variant of Hensel's Lemma. Let $p$ be a prime and let $\vec{a}\in\bbZ_p^n$. Suppose that $\vec{x}=\vec{a}$ is a solution to the system $f_i(\vec{x})=\vec{0}$ for which the Jacobian matrix $\brax{\partial f_i(\vec{x})/\partial x_j}_{ij}$ is nonsingular. Possibly after permuting the variables $x_i$ if necessary, we can assume that the submatrix $M(\vec{x})$ consisting of the last $R$ columns of $\brax{\partial f_i(\vec{x})/\partial x_j}_{ij}$ is nonsingular at $\vec{x}=\vec{a}$. 
	
	The so-called valuation theoretic Implicit Function Theorem then applies to the polynomials $f_i$ with the common zero $\vec{a}$ over the valued field $\bbQ_p$. This is essentially a version of Hensel's Lemma; see Kuhlmann~\cite[Theorem~25]{kuhlmannHensel}. If we write $\abs{\det M(\vec{a})}_p = p^{-\alpha}$, the theorem states that for all $p$-adic numbers $a'_1,\dotsc,a'_{n-R}\in \bbQ_p$ with $\abs{a'_i-a_i}_p < p^{-2\alpha}$, there are unique $p$-adic numbers $a'_{n-R+1},\dotsc,a'_{n}\in\bbQ_p$ with $\abs{a'_i-a_i}_p < p^{-\alpha}$ such that each $f_i(\vec{a}')=0$. 
	
	Now let $a'_1,\dotsc,a'_{n-R}$ be $p$-adic integers satisfying $a'_i \equiv a_i$ modulo $p^{2\alpha+1}$. For each $k\in\bbN$ there are $p^{(k-2\alpha-1)(n-R)}$ choices for $a'_i$ which are distinct modulo $p^k$, and by the theorem above each one extends to a vector of $p$-adic integers $\vec{a}$ satisfying $\vec{f}(\vec{a}')=0$.
	
	If this holds for each prime $p$, then $\singSeries$ is positive. For then reducing the vectors $\vec{a}'$  modulo $p^k$ gives $\gg_{\vec{f},p} p^{k(n-R)}$  distinct vectors $\vec{b}\in\set{1,\dotsc,p^k}^n$ satisfying the system of congruences $f_i(\vec{b})\equiv \vec{0}$ modulo $p^k$. The equality \eqref{1.eqn:evaluating_frakS} then shows that $\singSeries>0$.
\withqed\end{proof}

\begin{proof}[Proof of Theorem~\ref*{1.thm:main_thm_short}]
	We let $\cancellation = \frac{n-R+1}{4}$, and apply Theorem~\ref*{1.thm:manin} to the system of forms $F_i$. The result will follow if we can show that \eqref{1.eqn:aux_ineq_bound_in_manin_thm} holds, which is to say that
	\begin{equation}\label{1.eqn:aux_ineq_bound}
	\auxIneqOfSomethingNumSolns{\betaDotCapitalF}
	(B)
	\ll
	B^{\sigma_\bbR}
	\end{equation}
	for all $\vec{\beta}\in\bbR^R$ and all $B\geq 1$. Here the quantity $\sigma_\bbR$ is defined by \eqref{1.eqn:def_of_sigma-sub-R}.
	
	For each $\vec{\beta}\in\bbR^R$, let the matrix of the quadratic form $\betaDotCapitalF$ be $M\brax{\vec{\beta}}$. That is, $M\brax{\vec{\beta}}$ is the unique real $n\times n$ symmetric matrix with
	\[
	\betaDotCapitalF(\vec{x})
	=
	\vec{x}^T M\brax{\vec{\beta}}\vec{x}.
	\]
	Then we have
	\[
	\gradSomethingMultilinear{\betaDotCapitalF}{\vec{u}}
	=
	2M\brax{\vec{\beta}}\vec{u},
	\]
	so $\auxIneqOfSomethingNumSolns{\vec{\beta}}\brax{B}$ counts vectors $\vec{u}\in \bbZ^n$ satisfying
	\begin{equation*}
	\supnorm{\vec{u}}
	\leq B,
	\qquad
	\supnorm{M\brax{\vec{\beta}}}
	\leq
	\tfrac{1}{2}\supnorm{\betaDotCapitalF}.
	\end{equation*}
	These vectors $\vec{u}$ are all contained in the box $\supnorm{\vec{u}}\leq B$, and in the ellipsoid
	\begin{equation*}
	E(\vec{\beta})
	=	
	\set[]{\vec{\tee}\in\bbR^n \suchthat \vec{\tee}^T { M\brax{\vec{\beta}}^T M\brax{\vec{\beta}} } \vec{\tee}
		<
		n\cdot\supnorm{\betaDotCapitalF}^2}.
	\end{equation*}
	The ellipsoid has principal radii $\abs{\lambda}^{-1}\sqrt{n}\supnorm{\betaDotCapitalF}$ where $\lambda$ runs over the eigenvalues of the real symmetric matrix $M\brax{\vec{\beta}}$, counted with multiplicity. 
	Hence
	\begin{equation*}
	\auxIneqOfSomethingNumSolns{\betaDotCapitalF}\brax{B}
	\ll_n
	\prod_{\lambda} \min\setbrax{ \abs{\lambda}^{-1}\supnorm{\betaDotCapitalF}+1,\, B}
	\end{equation*}
	where $\lambda$ is as before. So to prove \eqref{1.eqn:aux_ineq_bound} it suffices that $n-\sigma_\bbR$ of the $\lambda$ are of size $\abs{\lambda} \gg \supnorm{\betaDotCapitalF}$ at least.
	
	Suppose for a contradiction that this is false. Then there exists a sequence $\vecsuper{\beta}{i}\in\bbR^R$ such that at least $\sigma_\bbR+1$ of the eigenvalues of $M\brax{\vecsuper{\beta}{i}\cdot\vec{q}}$ satisfy $\lambda = o(\supnorm{\vecsuper{\beta}{i}\cdot\vec{F}})$. By passing to a subsequence, we can assume $\vecsuper{\beta}{i}/\supnorm{\vecsuper{\beta}{i}} \to \vec{\beta}$, and then at least $\sigma_\bbR+1$ of the eigenvalues of $M\brax{\betaDotCapitalF}$ must be zero. In other words,
	\begin{equation*}
	\dim \sing V(\betaDotCapitalF)
	\geq 
	\sigma_\bbR.
	\end{equation*}
	But this contradicts the definition \eqref{1.eqn:def_of_sigma-sub-R}. So \eqref{1.eqn:aux_ineq_bound} holds as claimed.		
\withqed\end{proof}

As alluded to after Lemma~\ref*{1.lem:nonsing_case}, the argument used to prove Theorems~\ref*{1.thm:main_thm_short} and~\ref*{1.thm:manin} also yields weak approximation for $V(F_1,\dotsc,F_R)$ if that variety is smooth. It suffices to show that if the system $F_i(q\vec{x}-\vec{a})=0$ has solutions in the $p$-adic integers for each $p$, then it has integral solutions $\vec{x}$ with $\frac{\vec{x}}{\supnorm{\vec{x}}}$ arbitrarily close to $\frac{\vec{r}}{\supnorm{\vec{r}}}$, for any fixed real solution $\vec{r}$ to the system $F_i(\vec{r})=0$. For this one can let $\weightbox$ be a sufficiently small box containing $\frac{\vec{r}}{\supnorm{\vec{r}}}$, and repeat the proof of Theorems~\ref*{1.thm:main_thm_short} and \ref*{1.thm:manin} with the choice ${f}_i(\vec{x}) = F_i(q\vec{x}-\vec{a})$  instead of $f_i=F_i$ at the start of the proof of  Theorem~\ref*{1.thm:manin}. Since $\auxIneqNumSolns(B)=\auxIneqOfSomethingNumSolns{\betaDotCapitalF}(B)$ we obtain \eqref{1.eqn:HL_formula_application} as before. Recalling that any real or $p$-adic point of $V(F_1,\dotsc,F_R)$ must be smooth, the argument to prove that  $\singIntegralBox,\singSeries$ are positive goes through and we obtain the existence of an integral solution of the required kind.

\begin{proof}[Proof of Lemma~\ref*{1.lem:nonsing_case}]
	We prove the first inequality in \eqref{1.eqn:nonsing_case}. Let $\vec{\beta}\in\bbR^R\setminus\set{\vec{0}}$ such that
	\[
	\sigma_\bbR
	=\dim V(\betaDotCapitalF).
	\]
	Without loss of generality we may suppose that $\beta_R $ is nonzero. Then we have
	\[
	V(F_1,\dotsc,F_R)
	=
	V(F_1,\dotsc,F_{R-1},\betaDotCapitalF).
	\]
	Since $V(F_1,\dotsc,F_{R-1})$ has dimension $n-1-R$, it follows that
	\[
	V(F_1,\dotsc,F_{R-1})
	\cap
	\sing V(\betaDotCapitalF)
	\subset
	\sing V(F_1,\dotsc,F_R)
	\]
	and so 
	$
	V(F_1,\dotsc,F_{R-1})
	\cap
	\sing V(\betaDotCapitalF)
	=\emptyset,
	$
	as $V(F_1,\dotsc,F_{R})$ is smooth. It follows that $\dim \sing V(\betaDotCapitalF)\leq R-1$, which proves the first inequality in \eqref{1.eqn:nonsing_case}.
	
	The second inequality in \eqref{1.eqn:nonsing_case} follows from the work of Browning and Heath-Brown~\cite{browningHeathBrownDiffDegrees}. In those authors' formula (1.3), set
	\begin{gather*}
	D=2,
	\quad
	r_1=0,
	\quad
	r_2=R,
	\quad
	F_{i,2} =F_i.
	\end{gather*}
	Now the $R\times n$ Jacobian matrix $ \brax{\partial F_i(\vec{x})/\partial x_j}_{ij}$ has full rank at every nonzero solution $\vec{x}\in\widebar{\bbQ}^n$ to $F_1(\vec{x})=\dotsb=F_R(\vec{x})=0$, because $V(F_1,\dotsc,F_R)$ is smooth of dimension $n-1-R$. This makes $F_{i,j}$ a `nonsingular system" in the sense of Browning and Heath-Brown, as defined in their formula (1.7). The next step is to replace $F_{i,d}$ with an ``equivalent optimal system". The comments after formula (1.7) of those authors show that in our case this means replacing $F_i$ with $\sum_j A_{ij}f_j$, where $A$ is an invertible linear transformation. In particular this preserves $V(F_1,\dotsc,F_R)$ and $W$. Now their formulae (1.4) and (1.8) show that $
	B_2
	\leq
	R-1
	$, where
	$
	B_2 = 1+ \dim (W).
	$
	This proves \eqref{1.eqn:nonsing_case}.
\withqed\end{proof}

\subsection*{Acknowledgements}
This paper is based on a DPhil thesis submitted to Oxford University. I would like to thank my DPhil supervisor, Roger Heath-Brown. I am grateful to Victor Beresnevich and Bryan Birch for useful conversations on the topics discussed here, and to Ben Green and Shuntaro Yamagishi for their comments on earlier versions of this paper.

\bibliography{systems-of-many-forms}

\end{document}